\documentclass[a4paper,10pt]{article}
\usepackage[english]{babel}
\usepackage[latin1]{inputenc}
\usepackage[T1]{fontenc}
\usepackage{textcomp} 
\usepackage{amsmath}
\usepackage{amsfonts}
\usepackage{amsthm}
\usepackage{mathrsfs}
\usepackage{amssymb}
\usepackage{csquotes}
\usepackage[mathscr]{euscript}
\DeclareSymbolFont{rsfs}{U}{rsfs}{m}{n}
\DeclareSymbolFontAlphabet{\mathscrsfs}{rsfs}

\theoremstyle{definition}
\newtheorem{Def}{Definition}[section]

\newtheorem{Rmk}[Def]{Remark}

\theoremstyle{plain}
\newtheorem{Prop}[Def]{Proposition}
\newtheorem{Thm}[Def]{Theorem}
\newtheorem{Lemma}[Def]{Lemma}
\newtheorem{Cor}[Def]{Corollary}

\newcommand{\dbline}[2]{\genfrac{}{}{0pt}{}{#1}{#2}}

\newcommand{\R}{\mathbb{R}}

\newcommand{\Rd}{\R^d}

\newcommand{\Z}{\mathbb{Z}}
\newcommand{\N}{\mathbb{N}_0}

\newcommand{\Zd}{\Z^d}

\newcommand{\lp}{\ell}

\newcommand{\loc}{\text{\rm loc}}
\DeclareMathOperator{\supp}{supp}
\newcommand{\Lp}{L}
\newcommand{\leb}{\mathcal{L}}
\newcommand{\Lpl}{\Lp_\loc}
\newcommand{\ep}{T}
\newcommand{\Os}{\mathcal{O}}
\newcommand{\ha}{\dim_{\mathcal{H}}}
\newcommand{\ms}{D}
\newcommand{\dist}{\text{\rm dist}}

\newcommand{\fl}[1]{\lfloor #1 \rfloor}
\newcommand{\ce}[1]{\lceil #1 \rceil}
\newcommand{\seq}[1]{\boldsymbol{#1}}
\newcommand{\os}{\overline{s}}
\newcommand{\us}{\underline{s}}
\newcommand{\calD}{\mathscrsfs{D}}
\newcommand{\sign}{\text{\rm sign}}
\newcommand{\K}{\mathcal{K}}

\newcommand{\cs}{\ensuremath{\mathcal{S}}}
\newcommand{\ft}{\ensuremath{\mathcal{F}}}

\renewcommand{\phi}{\varphi}
\renewcommand{\epsilon}{\varepsilon}

\title{Generalized spaces of pointwise regularity:\\ To a general framework for the WLM}
\author{L. Loosveldt\footnote{Universit\'e de Li\`ege, D\'epartement de math\'ematique -- zone Polytech 1, 12 all\'ee de la D\'ecouverte, B\^at. B37, B-4000 Li\`ege} and S. Nicolay}

\begin{document}

\maketitle

\begin{abstract}
In this work we generalize the spaces $T^p_u$ introduced by Calder\'on and Zygmund using a pointwise version of conditions defining the generalized Besov spaces and give conditions binding the functions belonging to these spaces and the wavelet coefficients of such functions. Next, we propose a multifractal formalism based on the new spaces wich generalize the so-called wavelet leaders method and show that it is satisfied on a prevalent set.
\end{abstract}

\section{Introduction}
The H\"older-regularity can be seen as a notion that fills gaps between being $n$ times continuously differentiable and $n+1$ times continuously differentiable. More precisely, a function $f$ from $\Lpl^p(\Rd)$ belongs to the space $T^p_u(x_0)$ (with $x_0\in \Rd$, $p\in [1,\infty]$ and $u>0$) if there exists a polynomial $P_{x_0}$ of degree strictly less than $u$ such that
\begin{equation}\label{eq:def hol}
 r^{-u} \|f- P_{x_0} \|_{\Lp^p(B(x_0,r))} \le C,
\end{equation}
for $r>0$, where $B(x_0,r)$ denotes the open ball centered at $x_0$ with radius $r$ (see \cite{Calderon:61}); $T^\infty_u(x_0)$ is called a H\"older space and is usually denoted by $\Lambda^u(x_0)$ \cite{Krantz:83}. These spaces are embedded and the H\"older exponent of $f$ at $x_0$ is defined as
\begin{equation}\label{eq:Hexp}
 h_\infty (x_0) := \sup \{u>0 : f\in T^\infty_u(x_0) \}.
\end{equation}

The discrete wavelet transform provides a useful tool for studying the H\"older space (for more details, see \cite{Jaffard:04}), since the condition on $f$ at $x_0$ can be transposed to a condition on some wavelet coefficients near $x_0$, the so-called wavelet leaders (see Definition~\ref{def:wleaders} with $p=\infty$). Indeed, if a function belongs to a space $\Lambda^u(x_0)$, the wavelet leaders of $x_0$ satisfy an inequality somehow similar to (\ref{eq:def hol}). Conversely, if this condition on the wavelet leaders is met, the corresponding function belongs to a space close to $T^\infty_u(x_0)$; more precisely, one has
\begin{equation}\label{eq:def hol gen}
 \theta_u^{-1}(r) \|f- P_{x_0} \|_{\Lp^\infty(B(x_0,r))} \le C,
\end{equation}
with $\theta_u(r) = r^u |\ln(r)|$. In other words, $f$ belongs to $T^\infty_u(x_0)$ up to a logarithmic correction. If such a result holds, we will say that we have a quasi-characterization of the space; let us remark that such a quasi-characterization provides an exact characterization of the H\"older-regularity, i.e.\ of the H\"older exponent $h_\infty(x_0)$.

This notion of regularity can be generalized in several ways. First, one can replace the expression $r^{-u}$ appearing in~(\ref{eq:def hol}) with a function $\theta_u(r)$ satisfying some requirements, as in inequality~(\ref{eq:def hol gen}). The space of functions satisfying (\ref{eq:def hol gen}) has been studied in \cite{Kreit:18}, where a quasi-characterization is obtained. One can also replace the H\"older space appearing in (\ref{eq:Hexp}) with a $T^p_u$ space, in order to study non-locally bounded functions. This approach has been undertaken in \cite{Jaffard:05}, where generalized wavelet leaders, called $p$-leaders, are introduced. However, this definition is not a direct generalization of the usual leaders and fails to quasi-characterize the $T^p_u(x_0)$ spaces, although they still can be used to study the corresponding generalized H\"older exponent. The first part of this paper consists in combining these two points of view, by considering the spaces of functions satisfying the condition
\begin{equation}
 \theta_u^{-1}(r) \|f- P_{x_0} \|_{\Lp^p(B(x_0,r))} \le C.
\end{equation}
Indeed, we consider an even larger class of spaces called spaces of generalized pointwise smoothness (see Definition~\ref{def:gensps}) corresponding, in some way, to a pointwise version of the generalized Besov spaces (see \cite{Loosveldt:19} and references therein). We obtain a quasi-characterization of such spaces by introducing a variant definition of the $p$-leaders that naturally extends the classical case where $p=\infty$.

The second part of this paper aims at providing a multifractal formalism suited for the spaces introduced here. A multifractal formalism is an empirical method that allows to estimate the quantity
\[
 \ha \{x_0\in \Rd: h_p(x_0)= h\},
\]
where $\ha$ denotes the Hausdorff dimension with the convention $\ha(\emptyset)=-\infty$ (see \cite{Falconer:03} for example) and $h_p(x_0)$ is the generalized H\"older obtained by replacing $T^\infty_u(x_0)$  with $T^p_u(x_0)$ in (\ref{eq:Hexp}). Usually, one requires such a method to be valid for a large class of functions. Such a multifractal formalism was first presented in \cite{Parisi:85} in the context of the analysis of fully developed turbulence velocity data. We show here that, from the prevalence point of view, almost every function belonging to a space of generalized smoothness satisfies a multifractal formalism derived from the formalism relying on the wavelet leaders; in other words, the generalized Besov spaces provide a natural framework for supporting this theory.

This paper can be seen as a generalization of the ideas and techniques employed in \cite{Jaffard:04b,Jaffard:05b,Fraysse:06,Kreit:18}.

The notations used here are rather standard. Throughout this paper, we will use Euler's notation for the derivatives, i.e.\ $D_j f$ designates the derivative of $f$ following the $j$-th component.

\section{Generalized spaces of pointwise smoothness}

\subsection{Admissible sequences}
Let us recall the notion of admissible sequence (see e.g.\ \cite{Kreit:12} and references therein).
\begin{Def}
A sequence $\seq{\sigma}=(\sigma_j)_j$ of real positive numbers is called admissible if there exists a positive constant $C$ such that
\[
 C^{-1} \sigma_j \le \sigma_{j+1} \le C \sigma_j,
\]
for any $j\in\N$.
\end{Def}
If $\seq{\sigma}$ is such a sequence, we set
\[
 \underline{\sigma}_j=\inf_{k\in\N} \frac{\sigma_{j+k}}{\sigma_k}
 \quad\text{and}\quad
 \overline{\sigma}_j=\sup_{k\in\N} \frac{\sigma_{j+k}}{\sigma_k}
\]
and define the lower and upper Boyd indices as follows,
\[
 \us(\seq{\sigma})=\lim_j \frac{\log_2 \underline{\sigma}_j}{j}
 \quad\text{and}\quad
 \os(\seq{\sigma})=\lim_j \frac{\log_2 \overline{\sigma}_j}{j}.
\]
Since $(\log \underline{\sigma}_j)_j$ is a subadditive sequence, such limits always exist. The following relations about such sequences are well known (see e.g.\ \cite{Kreit:12}). If $\seq{\sigma}$ is an admissible sequence, let $\epsilon>0$; there exists a positive constant $C$ such that
\begin{equation*}\label{eq:boyd}
 C^{-1} 2^{j(\us(\seq{\sigma}) -\epsilon)}
 \le \underline{\sigma}_j \le \frac{\sigma_{j+k}}{\sigma_k} \le \overline{\sigma}_j \le
 C 2^{j(\os(\seq{\sigma})+\epsilon)},
\end{equation*}
for any $j,k\in\N$.
In this paper, $\seq{\sigma}$ will always stand for an admissible sequence and, given $u>0$, we set $\seq{u}= (2^{ju})_j$. Of course, we have $\us(\seq{u})= \os(\seq{u}) =u$.

\subsection{Definition of the generalized spaces of pointwise smoothness}
\begin{Def}\label{def:gensps}
Let $p,q\in [1,\infty]$, $f\in \Lpl^p$ and $x_0\in \Rd$; $f$ belongs to $\ep_{p,q}^{\seq{\sigma}}(x_0)$ whenever
\[
 (\sigma_j 2^{j d/p} \sup_{|h|\le 2^{-j}} \| \Delta_h^{\fl{\os(\sigma)}+1} f\|_{\Lp^p(B_h(x_0,2^{-j}))})_j \in \lp^q,
\]
where, given $r>0$,
\[
 B_h(x_0,r) =\{ x: [x,x+(\fl{\os(\seq{\sigma})}+1)h] \subset B(x_0,r)\}.
\]
\end{Def}
It is easy to check that $\ep_{\infty,\infty}^{\seq{\sigma}}(x_0)$ is the generalized H\"older space $\Lambda^{\seq{\sigma}}(x_0)$ introduced in \cite{Kreit:18}. These spaces can also be seen as a generalization of the spaces $T^p_u(x_0)$ introduced by Calder{\'o}n and Zygmund in \cite{Calderon:61}, as Corollary~\ref{pro:rel cal} will show.

Let us give an alternative definition of $\ep_{p,q}^{\seq{\sigma}}(x_0)$.
\begin{Prop}\label{prop13}
Let $p,q\in [1,\infty]$, $f\in \Lpl^p$, $x_0\in \Rd$ and $\seq{\sigma}$ be an admissible sequence such that $\us(\seq{\sigma})>0$. We have $f\in \ep_{p,q}^{\seq{\sigma}}(x_0)$ if and only if there exists a sequence of polynomials $(P_{j,x_0})_j$ of degree less or equal to $\fl{\os(\seq{\sigma})}$ such that
\begin{equation}\label{eq:def:pol}
 (\sigma_j 2^{j d/p} \| f- P_{j,x_0} \|_{\Lp^p(B(x_0,2^{-j}))})_j \in \lp^q.
\end{equation}
\end{Prop}
\begin{proof}
The necessity of the condition being a consequence of the Whitney Theorem, let us check the sufficiency. Let $j\in\N$; for any polynomial $P$ of degree less or equal to $n:= \fl{\os(\seq{\sigma})}$, we have, given $x,h\in \Rd$,
\[
 |\Delta_h^{n+1} f(x)|
 \le |\Delta_h^{n+1} \big( f(x) -P(x) \big)|
 \le C_n \sum_{k=0}^{n+1} | f(x+kh) - P(x+kh)|,
\]
for a constant $C_n$. Therefore, for $|h|\le 2^{-j}$ and $x\in B_h(x_0,2^{-j})$, we get
\[
 \| \Delta_h^{n+1} f\|_{\Lp^p(B_h(x_0,2^{-j}))}
 \le C_n (n+2) \| f-P \|_{\Lp^p(B(x_0,2^{-j}))},
\]
hence the conclusion.
\end{proof}

\subsection{Independence of the polynomial from the scale}
Under some additional assumptions on the admissible sequence $\seq{\sigma}$, the sequence of polynomials $(P_{j,x_0})_j$ appearing in inequality (\ref{eq:def:pol}) can be replaced by a unique polynomial $P_{x_0}$ independent from the scale $j$: $P_{x_0}= P_{j,x_0}$.

We first need some preliminary results. Let us first state a somehow standard result about inequalities on polynomials; we sketch a proof for the sake of completeness.
\begin{Lemma}
Given $x_0\in \Rd$, a radius $r>0$, $p\in(0,\infty]$ and a maximum degree $n$, there exist two constants $C,C'>0$ only depending on $p$ such that, for any polynomial $P$ of degree lower or equal to $n$,
\[
 \|D^\alpha P\|_{\Lp^p(B(x_0,r))} \le C r^{-|\alpha|} \|P\|_{\Lp^p(B(x_0,r))},
\]
for any multi-index $\alpha$
and
\[
 \sup_{x\in B(x_0,r)} |P(x)|
 \le C' r^{d/p} \| P \|_{\Lp^p(B(x_0,r))}.
\]
\end{Lemma}
\begin{proof}
For the first inequality, let us recall that the Markov inequality affirms that, given a convex bounded set $E$ of $\Rd$, there exists a constant $C_{E,p}> 0$ such that for any $n\in\N$ and $k\in \{1,\ldots,d\}$, we have
\[
 \|D_k P\|_{\Lp^p(E)} \le C_{E,p} (n+1)^2 \|P\|_{\Lp^p(E)},
\]
for any polynomial $P$ of degree less or equal to $n$. As a consequence, given $r>0$, there exists a constant $C>0$ depending on $n$ and $p$ such that, for any multi-index $\alpha$, we have
\[
 \|D^\alpha P\|_{\Lp^p(B(x_0,r))} \le C r^{-|\alpha|} \|P\|_{\Lp^p(B(x_0,r))}.
\]

That being done, using Sobolev's inequality, we can now write
\[
 \sup_{x\in B(x_0,r)} |P(x)|
 \le C' r^{d/p} \| P \|_{\Lp^p(B(x_0,r))},
\]
for a constant $C'>0$ which only depends on $n$ and $p$.
\end{proof}
\begin{Lemma}\label{lem:reste}
Let $m\in\N$, $\seq{\sigma}$ be an admissible sequence such that $\us(\seq{\sigma}^{-1})>m$ and $\seq{\epsilon} \in \lp^q$ with $q\in [1,\infty]$; there exists a sequence $\seq{\xi}\in \lp^q$ such that
\[
 \sum_{j=J}^\infty \epsilon_j 2^{jm} \sigma_j \le \xi_J 2^{Jm} \sigma_J,
\]
for all $J\in\N$.
\end{Lemma}
\begin{proof}
Let $\delta, \delta' >0$ be such that $-2\delta' >m + \os(\sigma)+ \delta$; given $J\in\N$, we have, using H\"older's inequality,
\begin{align*}
 \sum_{j=J}^\infty \epsilon_j 2^{jm} \sigma_j
 &\le C \sum_{j=J}^\infty \epsilon_j 2^{(j-J)(m+\os(\sigma)+\delta)} 2^{Jm} \sigma_J \\
 &\le C (\sum_{j=J}^\infty (\epsilon_j 2^{-\delta'(j-J)})^q)^{1/q} (\sum_{j=J}^\infty 2^{-p\delta'(j-J)})^{1/p} 2^{Jm} \sigma_J,
\end{align*}
where $p$ is the conjugate exponent of $q$ (with the usual modification if one of the indices is $\infty$). It remains to check that the sequence $\seq{\xi}$ defined by
\[
 \xi_j = C (\sum_{k=j}^\infty (\epsilon_j 2^{-\delta'(j-J)})^q)^{1/q}
\]
belongs to $\lp^q$, which is easy.
\end{proof}
In the same way, we can get the following result.
\begin{Lemma}\label{lem:reste2}
Let $m\in\N$, $\seq{\sigma}$ be an admissible sequence such that $\us(\seq{\sigma}^{-1})<m$ and $\seq{\epsilon} \in \lp^q$ with $q\in [1,\infty]$; there exists a sequence $\seq{\xi}\in \lp^q$ such that
\[
 \sum_{j=0}^J \epsilon_j 2^{jm} \sigma_j \le \xi_J 2^{Jm} \sigma_J,
\]
for all $J\in\N$.
\end{Lemma}
\begin{Rmk}
 Lemma~\ref{lem:reste} generalizes the relation $\sum_{j=J}^\infty \sigma_j \le C \sigma_J$, satisfied whenever $\us(\seq{\sigma}^{-1})>0$, while Lemma~\ref{lem:reste2} should be compared with $\sum_{j=0}^J 2^{jm} \sigma_j \le C 2^{Jm} \sigma_J$, holding for $\os(\seq{\sigma}^{-1})<m$ ($m\in\N$) (see e.g.\ \cite{Kreit:12} for more details).
\end{Rmk}

The main theorem of this section relies on the following lemma.
\begin{Lemma}
Let $p,q \in [1,\infty]$, $f\in \Lpl^p$, $x_0\in \Rd$ and $\seq{\sigma}$ be an admissible sequence such that $0\le n:= \fl{\os(\seq{\sigma})} < \us(\seq{\sigma})$. If $f$ belongs to $\ep^{\seq\sigma}_{p,q}(x_0)$, the sequence of polynomials $(P_{j,x_0})_j$ satisfying (\ref{eq:def:pol}) is such that, given a multi-index $\alpha$ for which $|\alpha|\le n$, there exists a sequence $\seq{\xi} \in \lp^q$ satisfying
\[
 2^{-|\alpha|j} \sigma_j |D^\alpha (P_{j,x_0} - P_{k,x_0}) (x_0) | \le \xi_j,
\]
whenever $j<k$.

In particular, under the same hypothesis, the sequence $(D^\alpha P_{j,x_0}(x_0))_j$ is Cauchy and its limit does not depend on the chosen sequence of polynomials satisfying (\ref{eq:def:pol}).
\end{Lemma}
\begin{proof}
Let $\seq{\epsilon} \in \lp^q$ be such that
\[
 \sigma_j 2^{j d/p} \| f- P_{j,x_0} \|_{\Lp^p(B(x_0,2^{-j}))} \le \epsilon_j,
\]
for any $j\in \N$. Given a multi-index $\alpha$ satisfying the hypothesis and $j\in \N$, we know that there exists a constant $C>0$ such that
\begin{align*}
 \lefteqn{\| D^\alpha (P_{j,x_0} - P_{j+1,x_0}) \|_{\Lp^p(B(x_0,2^{-(j+1)}))}} & \\
 &\le C2^{|\alpha| (j+1)} \|P_{j,x_0} - P_{j+1,x_0} \|_{\Lp^p(B(x_0,2^{-(j+1)}))} \\
 &\le C2^{|\alpha| (j+1)} \|P_{j,x_0} - f \|_{\Lp^p(B(x_0,2^{-(j+1)}))} + \|f - P_{j+1,x_0} \|_{\Lp^p(B(x_0,2^{-(j+1)}))} \\
 &\le C2^{|\alpha| (j+1)} (\epsilon_j 2^{j d/p} \sigma_j^{-1} + \epsilon_{j+1} 2^{(j+1) d/p} \sigma_{j+1}^{-1} ),
\end{align*}
which implies, from what we have obtained so far,
\[
 | D^\alpha (P_{j,x_0} - P_{j+1,x_0}) (x_0) | \le C' (\epsilon_j + \epsilon_{j+1}) 2^{|\alpha| j} \sigma_j^{-1}.
\]
For $j<k$, Lemma~\ref{lem:reste} then implies
\[
 | D^\alpha (P_{j,x_0} - P_{k,x_0}) (x_0) | \le \xi_j 2^{|\alpha| j} \sigma_j^{-1},
\]
for the right sequence $\seq{\xi} \in \lp^q$.

It remains to show that the limit $\calD^\alpha f(x_0)$ of the sequence $(D^\alpha P_{j,x_0}(x_0))_j$ is independent of the peculiar choice of the sequence $(D^\alpha P_{j,x_0}(x_0))_j$; let $(Q_{j,x_0})_j$ be another sequence of polynomials satisfying (\ref{eq:def:pol}). With the same reasoning as before, we get
\[
 | D^\alpha (P_{j,x_0} - Q_{j,x_0})(x_0) | \le C 2^{|\alpha| j} \sigma_j^{-1},
\]
for $j$ large enough, which is sufficient to assert that
\[
 |D^\alpha Q_{j,x_0} (x_0) - \calD^\alpha f(x_0)|
\]
tends to zero as $j$ tends to infinity.
\end{proof}
We are now able to show the existence of the unique polynomial $P_{x_0}$ introduced in the beginning of this section.
\begin{Thm}\label{pro:uniq pol}
Let $p,q\in [1,\infty]$, $f\in \Lpl^p$, $x_0\in\Rd$ and $\seq{\sigma}$ be an admissible sequence such that $0\le n:= \fl{\os(\seq{\sigma})}< \us(\seq{\sigma})$. The following assertions are equivalent:
\begin{itemize}
 \item $f$ belongs to $\ep^{\seq{\sigma}}_{p,q}(x_0)$,
 \item there exists a unique polynomial $P_{x_0}$ of degree less or equal to $n$ such that 
 \begin{equation}\label{eq:unic P}
  (\sigma_j 2^{j d/p} \| f- P_{x_0} \|_{\Lp^p(B(x_0,2^{-j}))})_j \in \lp^q.
 \end{equation}
\end{itemize}
\end{Thm}
\begin{proof}
We need to prove that the first assertion implies the second one. As $f$ belongs to $\ep^{\seq{\sigma}}_{p,q}(x_0)$, there exists a sequence of polynomials $(P_{j,x_0})_j$ of degree less or equal to $n$ such that
\[
 (\sigma_j 2^{j d/p} \| f- P_{j,x_0} \|_{\Lp^p(B(x_0,2^{-j}))})_j \in \lp^q.
\]
Given a multi-index $\alpha$ satisfying $|\alpha|\le n$, let us set
\[
 \calD^\alpha f(x_0) := \lim_j D^\alpha P_{j,x_0}(x_0)
\]
and define the polynomial
\begin{equation}\label{eq:def pol unic}
 P_{x_0} : x \mapsto \sum_{|\alpha|\le n} \calD^\alpha f(x_0) \, \frac{(x-x_0)^\alpha}{|\alpha|!}.
\end{equation}
One directly gets
\[
 \| P_{j,x_0} - P_{x_0} \|_{\Lp^p(B(x_0,2^{-j}))} \le \sum_{|\alpha|\le n} | D^\alpha P_{j,x_0} (x_0) - \calD^\alpha f(x_0) | \, 2^{-j (|\alpha|+ d/p)}.
\]
That being said, we know from the previous lemma that, given $\alpha$, there exists a sequence $\seq{\xi}^{(\alpha)} \in \lp^q$ such that
\[
 | D^\alpha P_{j,x_0}(x_0) - \calD^\alpha f (x_0) | \le \xi^{(\alpha)}_j 2^{|\alpha| j} \sigma_j^{-1}.
\]
We thus have
\[
 (\sigma_j 2^{j d/p} \| P_{j,x_0}- P_{x_0} \|_{\Lp^p(B(x_0,2^{-j}))})_j \in \lp^q,
\]
which proves the first part of the theorem.

Concerning the uniqueness of the polynomial, the idea of the proof is the same as the one given in \cite{Calderon:61} for the spaces $T^p_u(x_0)$. Let $P$ and $Q$ be two polynomials satisfying a relation of type (\ref{eq:unic P}); one directly gets $P(x_0)=Q(x_0)$. That being said, let us define
\[
 L:= \sum_{|\alpha|= m } c_\alpha (\cdot -x_0)^\alpha,
\]
 where $m$ is the lowest degree of $P-Q$, with
\[
 c_\alpha := \frac{D^\alpha (P-Q)(x_0)}{|\alpha|!}.
\]
If $m<\sup\{l\in \Z: l<\us(\seq{\sigma})\}$, one can write
\[
 \|L\|_{\Lp^1(B(x_0,1))} \le C (2^{-mj} \sigma_j^{-1} + 2^{-j}),
\]
for a constant $C$, which means $L=0$. For $m= \sup\{l\in \Z: l<\us(\seq{\sigma})\}$, we simply get $\|L\|_{L^1(B(x_0,1))} \le C 2^{-mj}\sigma_j^{-1}$, which implies $L=P-Q=0$.
\end{proof}
\begin{Rmk}\label{rem:unic pol}
 In the previous result, if $\seq{\sigma}$ is the usual sequence $\seq{u}$ with $u\in\N$, it is easy to check that the polynomial $P_{x_0}$ is unique if one requires its degree to be strictly smaller than $n$.
\end{Rmk}
\begin{Cor}\label{pro:rel cal}
Given $p\in [1,\infty]$, $u> 0$ and $x\in \Rd$, we have $T^{\seq{u}}_{p,\infty}(x_0) = T^p_u(x_0)$, where $T^p_u(x_0)$ denotes the class of functions $f\in \Lp^p$ such that there exists a polynomial $P_{x_0}$ of degree strictly less than $u$ with the property that
\[
 r^{-d/p} \| f - P_{x_0}\|_{\Lp^p(B(x_0,r))} \le C r^u,
\]
for a constant $C>0$.
\end{Cor}

\section{Wavelet criteria}
Various function spaces can be characterized by wavelets. We combine here the approaches adopted for the classical $T^p_u$ spaces, the generalized Besov spaces and the generalized pointwise H\"older spaces to obtain a ``nearly'' characterization of the spaces $\ep^{\seq{\sigma}}_{p,q}$.

\subsection{Definitions}
Let us briefly recall some definitions and notations about wavelets (for more precisions, see e.g.\ \cite{Daubechies:92,Meyer:95,Mallat:99}). Under some general assumptions, there exist a function $\phi$ and $2^d-1$
functions $(\psi^{(i)})_{1\le i<2^d}$, called wavelets, such that
\[
 \{\phi(x-k) : k\in\Zd\}\cup\{\psi^{(i)}(2^j x-k):1\le i<2^n, k\in \Zd, j\in\N \}
\]
form an orthogonal basis of $L^2$. Any function $f\in L^2$ can be decomposed as follows,
\[
 f(x)=\sum_{k\in \Zd} C_k \phi(x-k) + \sum_{j\in \N} \sum_{k\in\Zd} \sum_{1\le i<2^d} c^{(i)}_{j,k} \psi^{(i)}(2^j x-k),
\]
where
\[
c^{(i)}_{j,k}=2^{nj}\int_{\R^n}f(x) \psi^{(i)}(2^jx-k)\, dx
\]
and
\begin{equation}\label{eq:def Ck}
 C_k=\int_{\R^n} f(x) \phi(x-k)\, dx.
\end{equation}
Let us remark that we do not choose the $L^2$ normalization for the wavelets, but rather an $L^\infty$ normalization, which is better fitted to the study of the H\"olderian regularity.

Let $\lambda^{(i)}_{j,k}$ denote the dyadic cube
\[
 \lambda^{(i)}_{j,k} := \frac{i}{2^{j+1}} + \frac{k}{2^j} + [0, \frac{1}{2^{j+1}})^d.
\]
In the sequel, we will often omit any reference to the indices $i$, $j$ and $k$ for such cubes by writing $\lambda = \lambda^{(i)}_{j,k}$. We will also index the wavelet coefficients of a function $f$ with the dyadic cubes $\lambda$ so that $c_\lambda$ will refer to the quantity $c^{(i)}_{j,k}$. The notation $\Lambda_j$ will stand for the set of dyadic cubes $\lambda$ of $\Rd$ with side length $2^{-j}$ and the unique dyadic cube from $\Lambda_j$ containing the point $x_0 \in \Rd$ will be denoted $\lambda_j(x_0)$. The set of the dyadic cubes is $\Lambda:= \cup_{j\in\N} \Lambda_j$. Two dyadic cubes $\lambda$ and $\lambda'$ are adjacent if there exists $j\in\N$ such that $\lambda, \lambda' \in \Lambda_j$ and $\dist(\lambda,\lambda')=0$. The set of the $3^d$ dyadic cubes adjacent to $\lambda$ will be denoted by $3\lambda$.

As for the wavelet-based study of the pointwise H\"older spaces, we will work with wavelet leaders \cite{Jaffard:04b}. However, as we work here with $L^p$ norms, we need to introduce a generalized version.
\begin{Def}\label{def:wleaders}
Given a dyadic cube $\lambda \in \Lambda_j$ at scale $j$, the $p$-wavelet leader of $\lambda$ ($p\in [1,\infty]$) is defined by
\[
 d^p_{\lambda} = \sup_{j'\ge j} (\sum_{\lambda'\in \Lambda_{j'}, \lambda'\subset \lambda} (2^{(j-j')d/p} |c_{\lambda'}| )^p )^{1/p}.
\]
Given $x_0 \in \Rd$, we set
\[
d^p_j(x_0) = \sup_{\lambda \in 3 \lambda_j(x_0)} d^p_\lambda.
\]
\end{Def}
\begin{Rmk}
 The definition of the wavelet leaders given here is different from the one presented in \cite{Leonarduzzi:14}. The quantities introduced here are easier to work with and naturally generalize the usual wavelet leaders $d_j(x_0)$ introduced in \cite{Jaffard:04b}, since we have $d_j(x_0)= d^\infty_j(x_0)$.
\end{Rmk}

We will need the following definition (see \cite{Meyer:97}), ensuring a minimum regularity condition for a function.
\begin{Def}
 Given $x_0\in\Rd$, a function $f$ defined on $\Rd$ belongs to the Xu space $\dot{X}^u_{p,q}(x_0)$ ($u\in\R$, $p,q\in [1,\infty]$) if there exists a constant $C_* >0$ such that
 \[
  ( \sum_{|k- 2^j x_0|< C_* 2^j} (2^{(u-d/p)j} |c_{\lambda^{(i)}_{j,k}}| )^p )^{1/p} \in \lp^q.
 \]
\end{Def}

\subsection{Characterization using compactly supported wavelets}
In this section, we consider a compactly supported wavelet basis of regularity $r > \os(\seq{\sigma}^{-1})$; such wavelets are considered in \cite{Daubechies:88}. In this context, $j_0$ is a natural number such that the support of each wavelet is contained in $B(0,2^{j_0})$. 

Let us first give a necessary condition for a function to belong to $\ep^{\seq{\sigma}}_{p,q}(x_0)$.
\begin{Thm}\label{pro:wav1}
 If $f$ belongs to the space $\ep^{\seq{\sigma}}_{p,q}(x_0)$, then
 \[
  (\sigma_j d_j^p(x_0) )_j \in \lp^q.
 \]
\end{Thm}
\begin{proof}
 Let $\seq{\epsilon} \in \lp^q$ and $(P_j)_j$ be a sequence of polynomial of degree less or equal to $\os(\seq{\sigma})$ such that
 \[
  \sigma_j 2^{j d/p} \| f- P_j \|_{\Lp^p(B(x_0,2^{-j}))} \le \epsilon_j,
 \]
 for all $j\in\N$. Let us set choose $j_1\in \N$ such that $2\sqrt{d} +2^{j_0} \le 2^{j_1}$ and fix $n\ge j_1$. For $\lambda^{(i)}_{j,k} \subset 3\lambda_n(x_0)$, we have
 \[
  |\frac{k}{2^j} - x_0| \le 2\sqrt{d} 2^{-n}.
 \]
 By setting
 \[
  \Lambda_{j,n} := \{\lambda_{j,k}^{(i)} \in \Lambda_j : |k- 2^j x_0| \le 2 \sqrt{d} 2^{j-n} \},
 \]
 for $\lambda\in 3\lambda_n(x_0)$, we can write
 \[
  \sum_{\lambda'\in \Lambda_j, \lambda'\subset \lambda} 2^{(n-j)d} |c_{\lambda'}|^p \le \sum_{\lambda' \in \Lambda_{j,n}} 2^{(n-j)d} |c_{\lambda'}|^p,
 \]
 whenever $p\not= \infty$. Let us set
 \[
  s_{n,j} :=   \sum_{\lambda' \in \Lambda_{j,n}} |c_{\lambda'}|^p
 \]
 and define
 \[
  g_{n,j} :=  \sum_{\lambda'\in \Lambda_{j,n}} |c_{\lambda'}|^{p-1} \sign(c_{\lambda'}) \psi_{\lambda'}.
 \]
 One easily check the the support of $g_{n,j}$ is contained in $B(x_0, 2^{j_1-n})$ and
 \[
  s_{n,j} = 2^{jd} \langle f, g_{n,j} \rangle
  = 2^{jd} \int_{B(x_0, 2^{j_1-n})} ( f(x) - P_{n-j_1}(x) ) \overline{g_{n,j}(x)} \, dx,
 \]
 so that, if we denote by $q$ the conjugate exponent of $p$,
 \[
  s_{n,j} \le 2^{jd} \| f -  P_{n-j_1} \|_{\Lp^p(B(x_0, 2^{j_1-n}))} \, \|g_{n,j} \|_{\Lp^q}.
 \]

 To estimate $\|g_{n,j} \|_{\Lp^q}$, let us remark that there exists a constant $C_*>0$ that does not depend on $\lambda$ nor the scale $j$ such that the cardinal of
 \[
  \{ \lambda' \in \Lambda_j : \supp(\psi_\lambda) \cap \supp(\psi_{\lambda'}) \not= \emptyset \}
 \]
 is bounded by $C_*$. Therefore, given $j\in\N$, we can choose a partition $E_1,\ldots, E_{C_*}$ of $\Lambda_j$ such that $\lambda',\lambda''\in E_m$ ($1\le m\le C_*$) and
 \[
 \supp(\psi_{\lambda'}) \cap \supp(\psi_{\lambda''}) \not= \emptyset
 \]
 implies $\lambda' =\lambda''$. For $p\not= 1$, we easily get
 \[
  |g_{n,j}|^q \le C_*^q \sum_{\lambda' \in \Lambda_{j,n}} |c_{\lambda'}|^p |\psi_{\lambda'}|^q
 \]
 and thus
 \begin{equation}\label{eq:wav:p1}
  \|g_{n,j} \|_{\Lp^q} \le C_* 2^{-j d/q} s_{n,j}^{1/q} \max_{1\le i< 2^d} \|\psi^{(i)}\|_{\Lp^q}.
 \end{equation}
 If $p=1$, one easily checks that one has
 \[
  \|g_{n,j} \|_{\Lp^\infty} \le C_* 2^{-j d/q} \max_{1\le i< 2^d} \|\psi^{(i)}\|_{\Lp^\infty},
 \]
 so that (\ref{eq:wav:p1}) is still satisfied in this case.
 
 That being done, since we have
 \[
  s_{n,j}^{1/p}\le C \epsilon_{n-j_1} 2^{(j-n)d/p} \sigma_n^{-1},
 \]
 for a constant $C>0$, we get
 \[
  \sum_{\lambda'\in \Lambda_j, \lambda'\subset \lambda} 2^{(n-j)d} |c_{\lambda'}|^p \le 2^{(n-j)d} s_{n,j} \le C \epsilon_{n-j_1}^p \sigma_n^{-p},
 \]
 which is sufficient to conclude in the case $p\not= \infty$.

 Finally, let us consider the case $p=\infty$. Indeed, the conclusion is straightforward since, given $\lambda \subset 3\lambda_n(x_0)$, one easily check that, using an analogous reasoning, we can write
 \[
  |c_\lambda| \le C \epsilon_{n-j_1} \sigma_n,
 \]
 for a constant $C>0$.
\end{proof}

For the sufficient condition, we need the following definition.
\begin{Def}
 Let $p,q\in [1,\infty]$, $x_0\in \Rd$ and $f$ be a function from $\Lpl^p$; if $\seq{\sigma}$ is an admissible sequence such that $2^{-j d/p} \sigma_j^{-1}$ tends to $0$ as $j$ tends to $\infty$, we says that $f$ belongs to $\ep^{\seq{\sigma}}_{p,q,\log}(x_0)$ if there exists $J\in \N$ for which
 \[
  ( \frac{2^{j d/p} \sigma_j}{\log_2(2^{-j d/p} \sigma_j^{-1})} \sup_{|h|\le 2^{-j}} \| \Delta_h^{\fl{\os(\sigma^{-1})}+1} f \|_{\Lp^p(B_h(x_0,2^{-j}))} )_{j\ge J} \in \lp^q.
 \]
\end{Def}
\begin{Thm}\label{pro:wav2}
 Let $p,q\in [1,\infty]$, $x_0\in \Rd$ and $f$ be a function from $\Lpl^p$; let also $\seq{\sigma}$ be an admissible sequence such that $2^{-j d/p} \sigma_j^{-1}$ tends to $0$ as $j$ tends to $\infty$ and $\underline{\sigma}_1 >2^{-d/p}$. If $f$ belongs to $\dot{X}^\eta_{p,q}(x_0)$ for some $\eta>0$, then
 \[
  (\sigma_j d_j^p (x_0))_j \in \lp^q
 \]
 implies $f\in \ep^{\seq{\sigma}}_{p,q,\log}(x_0)$.
\end{Thm}
\begin{proof}
 Let us first suppose that $\os(\seq{\sigma}) \ge 0$ and set $n:= \fl{\os(\seq{\sigma})}$. We need to define some quantities. First, choose $m\in \N$ such that $k/2^j \in B(x,r)$ implies $\lambda^{(i)}_{j,k} \subset B(x, 2^m r)$, for any $x\in \Rd$, $k\in \Zd$, $j\in \N$ and $r\ge 2^{-j}$. Let also $m'\in \N$ be such that, for any $x\in\Rd$ and any $j\in\N$, $B(x,2^{-j})$ is included in some dyadic cube of side length $2^{m'-j}$ and define $J_0 := j_0 +m +m'$. Let $C_*>0$ be such that
 \[
  ( \sum_{|k - 2^j x_0| \le C_* 2^j} ( 2^{\eta- d/p)j} | c_{\lambda^{(i)}_{j,k}}| )^p )^{1/p} \in \lp^q
 \]
 and choose a number $J_1\in \N$ for which we have $(1+2^{j_0})\le C_* 2^{J_1}$. We also need a sequence $\seq{\epsilon}\in \lp^q$ satisfying $\sigma_j d_j^p(x_0) \le \epsilon_j$, for all $j\in \N$. Finally, given $J\ge \max\{J_0,J_1\}$, define
 \[
  P_J := \sum_{|\alpha|\le n} ( \frac{(\cdot- x_0)^\alpha}{|\alpha|!} \sum_{j=-1}^J D^\alpha f_j(x_0) ),
 \]
 where
 \[
  f_{-1} := \sum_{k\in \Zd} C_k \phi_k
  \qquad\text{and}\qquad
  f_j := \sum_{\lambda \in \Lambda_j} c_\lambda \psi_\lambda,
 \]
 for $j\ge 0$. We have
 \begin{align}
  \lefteqn{2^{J d/p} \| f- P_J \|_{\Lp^p(B(x_0,2^{-J}))}}& \nonumber \\
  &\le \sum_{j=-1}^J 2^{J d/p} \| f_j-  \sum_{|\alpha|\le n} \frac{(\cdot- x_0)^\alpha}{|\alpha|!} D^\alpha f_j(x_0) \|_{\Lp^p(B(x_0,2^{-J}))} \label{eq:wav1} \\
  & \quad + \sum_{j= J+1}^\infty  2^{J d/p} \| f_j \|_{\Lp^p(B(x_0,2^{-J}))}. \label{eq:wav2}
 \end{align}

 Let us fix $y\in B(x_0,2^{-J})$ and $|\alpha|= n+1$. We will first consider the case $p\not= \infty$. We have $D^\alpha \psi_{\lambda^{(i)}_{j,k}} (y) \not=0$ only if $k/2^j$ belongs to $B(y,2^{j_0-j})$; for $J_0\le j\le J$, we have
 \[
  \lambda^{(i)}_{j,k} \subset B(y,2^{m-j-j_0}) \subset \lambda_{j-J_0} (x_0),
 \]
 so that we can write, using the same reasoning as in the previous proof,
 \begin{align*}
  |D^\alpha f_j (y)|
  &\le C 2^{jp(n+1)} \sum_{\lambda \in \Lambda_j} |c_\lambda|^p |D^\alpha \psi_\lambda(y)|^p \\
  &\le C 2^{jp(n+1)} \sum_{\lambda \in \Lambda_j, \lambda\subset \lambda_{j-J_0}(x_0)} |c_\lambda|^p |D^\alpha \psi_\lambda(y)|^p \\
  &\le C 2^{jp(n+1)} \epsilon_{j-J_0}^p \sigma_j^{-p},
 \end{align*}
 since $\seq{\sigma}$ is an admissible sequence. Moreover, as the wavelet coefficients are finite 
 and there exists a constant $C_d$, which only depends on $d$, such that
 \[
  \# \{k\in \Z^d : k\in B(y, 2^{j_0}) \} \le C_d, \quad
  \# \{ k\in\Zd : k/2^j \in B(y,2^{j_0-j}) \} \le C_d,
 \]
 we also have
 \[
  | D^\alpha f_j (y)|^p \le C 2^{jp(n+1)} \sigma_j^{-p},
 \]
 for all $j\in \{-1,\ldots, J_0-1\}$. As a consequence, we can write, for any $j\in \{-1,\ldots, J\}$,
 \[
  \| f_j - \sum_{|\alpha|\le n} \frac{(\cdot- x_0)^\alpha}{|\alpha|!} D^\alpha f_j (x_0) \|_{\Lp^p(B(x_0,2^{-J}))}
  \le \theta_j 2^{-J(n+1+d/p)} 2^{j(n+1)} \sigma_j^{-1},
 \]
 for some sequence $\seq{\theta}\in \lp^q$. A similar reasoning gives the same inequality for $p=\infty$. Now, since $\os(\seq{\sigma}) <n+1$, (\ref{eq:wav1}) is upperbounded by
 \[
  C' 2^{-J(n+1)} \sum_{j=-1}^J \theta_j 2^{j(n+1)} \sigma_j^{-1}
  \le C' \xi_J \sigma_J^{-1},
 \]
 for some constant $C'>0$, where the sequence $\seq{\xi}$ is given by Lemma~\ref{lem:reste}.

 For the second term, let us fix $j\ge J+1$ and $p\not=\infty$ to define
 \[
  \Lambda_{j,J} := \{ \lambda^{(i)}_{j,k} \in \Lambda_j : B(k/2^j,2^{j_0}/2^j) \cap B(x_0,2^{-J}) \not=\emptyset \}.
 \]
 By proceeding as before for $x\in B(x_0, 2^{-J})$, we get
 \begin{align}
  \|f_j(x)\|_{\Lp^p(B(x_0,2^{-J}))}^p \le C \sum_{\lambda\in \Lambda_{j,J}} 2^{-dj} |c_\lambda|^p, \label{eq:wav3}
 \end{align}
 for some constant $C$, which gives
 \[
  2^{J d/p} \|f_j(x)\|_{\Lp^p(B(x_0,2^{-J}))} \le C \epsilon_{J-J_0} \sigma_J^{-1}.
 \]
 Moreover, since the coefficient $c_{\lambda^{(i)}_{j,k}}$ does not vanish in the sum~(\ref{eq:wav3}) only if the condition $|k-2^j x_0|\le C_* 2^j$ is satisfied, we also have
 \[
  \|f_j(x)\|_{\Lp^p(B(x_0,2^{-J}))}^p \le \delta_j^p 2^{-\eta pj},
 \]
 for a sequence $\seq{\delta} \in \lp^q$, as $f$ belongs to the space $\dot{X}^\eta_{p,q}(x_0)$. Let us obtain upperbounds when $p=\infty$; for $x\in B(x_0,2^{-J})$, $k/2^j \in B(x,2^{j_0-j})$ implies $\lambda^{(i)}_{j,k}\subset \lambda_{J-j_0}(x_0)$, so that we have $|c_{\lambda^{(k)}_{j,l}}|\le C \epsilon_{J-J_0} \sigma_N$. The same reasoning as before leads to
 \[
  \|f_j(x)\|_{\Lp^\infty(B(x_0,2^{-J}))} \le C \delta_j 2^{-\eta j}.
 \]

 Let us now set $j_*(J) := \ce{|\log_2(2^{-J d/p} \sigma_J^{-1})|/\eta}$ and choose $\eta$ small enough in order to insure that we have $\log_2(2^{d/p} \underline{\sigma}_1)/\eta >1$. With such a definition, we have $j_*(J)= j_*(J')$ if and only if $J=J'$ and we can write
 \begin{align*}
  \lefteqn{\sum_{j=J+1}^\infty 2^{J d/p} \| f_j \|_{\Lp^p(B(x_0,2^{-J}))}} & \\
  &= \sum_{j=J+1}^{j_*(J)} 2^{J d/p} \| f_j\|_{\Lp^p(B(x_0,2^{-J}))} + \sum_{j=j_*(J)+1}^\infty 2^{J d/p} \| f_j\|_{\Lp^p(B(x_0,2^{-J}))} \\
  &\le C \sum_{j=J+1}^{j_*(J)} \epsilon_{J-J_0} \sigma_J^{-1} + C 2^{J d/p} \sum_{j=j_*(J)+1}^\infty \delta_j 2^{-\eta j} \\
  &\le C (\epsilon_{J-J_0} + \xi_{j_1(J)}) \, |\log_2(2^{-J d/p} \sigma_J^{-1})| \sigma_J^{-1},
 \end{align*}
 for $J$ large enough, where the sequence $(\xi_{j_*(J)})_J$ belongs to $\lp^q$.

 It only remains to consider the situation where $\us(\seq{\sigma})<0$. In this case, let us set $P_J=0$ whenever $J\ge \max\{J_0,J_1\}$; once again, there exists a sequence $\seq{\xi}\in \lp^q$ such that
 \[
  |f_j(y)| \le \xi_j \sigma_J^{-1},
 \]
 for $y\in B(x_0,2^{-J})$, any $J\ge \max\{J_0,J_1\}$ and any $j\in \{-1,\ldots, J\}$. As done previously, we gets
 \begin{align*}
  \lefteqn{2^{J d/p} \| f -P_J \|_{\Lp^p(B(x_0,2^{-J}))}} & \\
  &\le C \sum_{j=-1}^J 2^{J d/p} \|f_j\|_{\Lp^p(B(x_0,2^{-J}))} +  \sum_{j=J+1}^\infty 2^{J d/p} \|f_j\|_{\Lp^p(B(x_0,2^{-J}))} \\
  &\le \delta_J |\log_2 (2^{-J d/p} \sigma_J^{-1})| \sigma_J^{-1},
 \end{align*}
 with $\seq{\delta}\in \lp^q$.
\end{proof}
\begin{Rmk}
 It is well known that Theorem~\ref{pro:wav1} has no converse: the ``logarithmic correction'' appearing in Theorem~\ref{pro:wav2} is necessary in the classical case (see e.g.\ \cite{Jaffard:04b}).
\end{Rmk}

\subsection{Characterization using wavelets in the Schwarz class}
In practise, compactly supported wavelets are used most often; however, for theoretical applications, in can be handy to have similar results concerning wavelets in the Schwarz class \cite{Meyer:95}. We will thus consider such wavelets in this section.
\begin{Lemma}
 Let $p,q\in [1,\infty]$, $x_0\in \Rd$ and $\seq{\sigma}$ be an admissible sequence such that either $\os(\seq{\sigma})>-d/p$, $\os(\seq{\sigma})<0$ or $0\le n:= \fl{\os(\seq{\sigma})} < \us(\seq{\sigma})$; if $f\in \Lp^p$ belongs to $\ep^{\seq{\sigma}}_{p,q}(x_0)$, then we have
 \[
  (\sigma_j 2^{j(d-u)} \int_{\Rd\setminus B(x_0,2^{-j})} \frac{|f(x)-P(x)|}{|x_0-x|^u} \, dx)_j \in \lp^q,
 \]
 for any $u> \os(\seq{\sigma})+d$, where $P$ is the polynomial given by Theorem~\ref{pro:uniq pol}.
\end{Lemma}
\begin{proof}
 Let us set $R:=f-P$; without loss of generality, we can assume $x_0=0$. Let us define, for $r>0$,
 \[
  \phi(r) := \int_{B(0,r)} |R(x)| \, dx;
 \]
 we know that there exists a sequence $\seq{\epsilon}\in \lp^q$ such that
 \[
  \phi(2^{-j}) \le 2^{-jd} \epsilon_j \sigma_j^{-1},
 \]
 for all $j\in\N$. Moreover, for $r\ge 1$, we have
 \[
  \phi(r) \le C r^{d(1-1/p)} \| f\|_{\Lp^p} + c r^{n+d} \le C r^{d+\os(\seq{\sigma})}.
 \]
 Using spherical coordinates, we can write
 \[
  \phi(r) = \int_0^r \psi(\rho) \, d\rho,
 \]
 with
 \[
  \psi(\rho) := \rho^{d-1} \int_0^{2\pi} \int_0^\pi \cdots \int_0^\pi |R(x(\rho,\theta_1,\ldots, \theta_{d-1}))| \, d\Omega_d,
 \]
 where $d\Omega_d$ stands for
 \[
  \sin^{d-2}(\theta_1) \cdots \sin(\theta_{d-2}) \, d\theta_1\cdots d\theta_{d-1}.
 \]
 Since, for all $r>0$, we have
 \[
  \frac{\phi(r)}{r^u} - \phi(2^{-j}) 2^{ju} = \int_{B(0,2^{-j})} \frac{|R(x)|}{|x|^u} \, dx - \int_{2^{-j}}^r \frac{u}{\rho^{u+1}} \, \phi(\rho) \, d\rho,
 \]
 we get
 \begin{align*}
  \lefteqn{\int_{B(0,r)\setminus B(0,2^{-j})} \frac{|R(x)|}{|x|^u} \, dx} & \\
  &\le \frac{\phi(r)}{r^u} + \int_1^r \frac{u}{\rho^{u+1}} \, \phi(\rho) \, d\rho + \sum_{k=1}^j \int_{2^{-k}}^{2^{1-k}} \frac{u}{\rho^{u+1}} \, \phi(\rho) \, d\rho.
 \end{align*}
 Since
 \[
  \phi(r)/r^u \le C \le C 2^{j(u-d)} 2^{-\delta j} \sigma_j^{-1},
 \]
 where $\delta>0$ has been chosen such that $\delta< u-d-\os(\seq{\sigma})$, we can write
 \[
  \int_1^r \frac{u}{\rho^{u+1}} \, \phi(\rho)\, d\rho \le C 2^{j(u-d)} 2^{-j \delta} \sigma_j^{-1}.
 \]
 Finally, as $\seq{\sigma}$ is admissible and $u>\os(\seq{\sigma})+d$, we have
 \[
  \sum_{k=1}^j \int_{2^{-k}}^{2^{1-k}} \frac{u}{\rho^{u+1}} \, \phi(\rho) \, d\rho \le 2^{j(u-d)} \xi_j \sigma_j^{-1},
 \]
 where $\seq{\xi}\in \lp^q$ is given by Lemma~\ref{lem:reste}. Putting all these informations together, we can claim that there exists a sequence $\seq{\theta}\in \lp^q$ such that the inequality
 \[
  \int_{B(0,r)\setminus B(0,2^{-j})} \frac{|R(x)|}{|x|^u} \, dx \le 2^{j(u-d)} \theta_j \sigma_j^{-1}
 \]
 holds for $r\ge 1$.
\end{proof}
\begin{Thm}
  Let $p,q\in [1,\infty]$, $x_0\in \Rd$ and $\seq{\sigma}$ be an admissible sequence such that either $\os(\seq{\sigma})>-d/p$, $\os(\seq{\sigma})<0$ or $0\le \fl{\os(\seq{\sigma})} < \us(\seq{\sigma})$; if $f\in \Lp^p$ belongs to $\ep^{\seq{\sigma}}_{p,q}(x_0)$, then we have
 \[
  (\sigma_j d_j^p(x_0))_j \in \lp^q.
 \]
\end{Thm}
\begin{proof}
 Let $\seq{\epsilon}\in \lp^q$ be such that
 \[
  \sigma_j 2^{j d/p} \|f- P_j\|_{\Lp^p(B(x_0,2^{-j}))} \le \epsilon_j,
 \]
 for any $j\in\N$, choose $j_1\in \N$ such that $2\sqrt{d}\le 2^{j_1}$ and fix $n\ge j_1+1$.

 Let us first suppose that $p\in (1,\infty)$; define
 \[
  \Lambda_{j,n} := \{ \lambda_{j,k}^{(l)} \in \Lambda_j : |k- 2^j x_0| \le \sqrt{d} 2^{j+1-n}\},
 \]
 so that $\lambda\in 3 \lambda_n(x_0)$ and $\lambda\in \Lambda_j$ implies $\lambda\in \Lambda_{j,n}$,
 \[
  s_{j,n} := \sum_{\lambda' \in \Lambda_{j,n}} |c_{\lambda'}|^p
 \]
 and
 \[
  g_{j,n} :=  \sum_{\lambda' \in \Lambda_{j,n}} |c_{\lambda'}|^{p-1} \sign(c_{\lambda'}) \psi_{\lambda'}.
 \]
 We have
 \begin{align*}
  s_{j,n}
  &= 2^{jd} \int_{B(x_0,2^{j_1-n+1})} (f(x)- P(x)) \overline{g_{j,n}(x)} \, dx \\
  &\quad + 2^{jd} \int_{\Rd\setminus B(x_0,2^{j_1-n+1})} (f(x)- P(x)) \overline{g_{j,n}(x)} \, dx.
 \end{align*}
 Using H\"older's inequality, we can write
 \begin{align*}
  \lefteqn{2^{jd} \int_{B(x_0,2^{j_1-n+1})} (f(x)- P(x)) \overline{g_{j,n}(x)} \, dx} & \\
  &\le C \epsilon_{n-j_1-1} 2^{jd} 2^{-n d/p} \| g_{j,n}\|_{\Lp^{p'}} \sigma_n^{-1},
 \end{align*}
 where $p'$ is the conjugate exponent of $p$, with
 \[
  \| g_{j,n}\|_{\Lp^{p'}} \le C 2^{-j d/p'} s_{j,n}^{1/p'},
 \]
 for a constant $C>0$, thanks to the wavelet characterization of the $\Lp^p$ spaces (see e.g.\ \cite{Meyer:95}). Now, for all $u> d+\os(\seq{\sigma})$, it is easy to check, thanks to the fast decay of the wavelets, that there exists a constant $C_{d,u}>0$ such that, for all $x\in \Rd\setminus B(x_0,2^{j_1-n+1})$,
 \[
  (\sum_{\lambda' \in \Lambda_{j,n}} |\psi_{\lambda'}|^p)^{1/p}
  \le C_{d,u}/ (2^j |x-x_0|)^u.
 \]
 Using the previous lemma, we can claim that there exists a sequence $\seq{\theta}\in \lp^q$ for which
 \[
  2^{jd} \int_{\Rd\setminus B(x_0,2^{j_1-n+1})} (f(x)- P(x)) \overline{g_{j,n}(x)} \, dx
  \le \theta_n s_{j,n}^{1/p'} 2^{(j-n)d/p} \sigma_n^{-1}.
 \]
 As a consequence, there exists a sequence $\seq{\xi}\in \lp^q$ such that
 \[
  s_{j,n}^{1/p} \le \xi_n 2^{(j-n)d/p} \sigma_n^{-1}.
 \]

 If $p=1$, keeping the same notations, we have
 \begin{align*}
  s_{j,n}
  &\le 2^{jd} \int_{B(x_0,2^{j_1-n+1})} |f(x)- P(x)| \sum_{\lambda' \in \Lambda_{j,n}} |\psi_{\lambda'}| \, dx \\
  &\quad + 2^{jd} \int_{\Rd\setminus B(x_0,2^{j_1-n+1})} |f(x)- P(x)| \sum_{\lambda' \in \Lambda_{j,n}} |\psi_{\lambda'}| \, dx.
 \end{align*}
 To bound the first integral, remark that $\sum_{\lambda' \in \Lambda_j} |\psi_{\lambda'}|$ is bounded and
 \[
  \| f-P \|_{\Lp^1(B(x_0,2^{j_1-n+1}))} \le C \epsilon_{n-j_1-1} 2^{-nd} \sigma_n^{-1}.
 \]
 The second integral can be treated as in the case $p\in (1,\infty)$.

 Finally, suppose that $p=\infty$, fix $j\ge n$ and suppose that $\lambda\in \Lambda_j$ satisfies $\lambda\in 3 \lambda_n(x_0)$. We have
 \begin{align*}
  |c_\lambda|
  &\le 2^{jd} \int_{B(x_0,2^{j_1-n+1})} |f(x)- P(x)| \, |\psi_{\lambda}| \, dx \\
  &\quad + 2^{jd} \int_{\Rd\setminus B(x_0,2^{j_1-n+1})} |f(x)- P(x)| \, |\psi_{\lambda}| \, dx.
 \end{align*}
 Once again, it is sufficient to bound the first integral, which is easy since we have
 \[
  2^{jd} \int_{B(x_0,2^{j_1-n+1})} |f(x)- P(x)| \, dx
  \le C \epsilon_{n-j_1-1} \sigma_n^{-1},
 \]
 for some constant $C>0$.
\end{proof}

\begin{Thm}
 Let $p,q\in [1,\infty]$, $f\in \Lpl^p$, $x_0\in \Rd$ and $\seq{\sigma}$ be an admissible sequence such that $\us(\seq{\sigma})> -d/p$ and $\underline{\sigma}_1 >2^{-d/p}$. If there exists $\eta>0$ such that $f\in B^\eta_{p,q}$, then
 \[
  (\sigma_j d_j^p(x_0) )_j \in \lp^q
 \]
 implies $f\in \ep^{\seq{\sigma}}_{p,q,\log} (x_0)$.
\end{Thm}
\begin{proof}
 Let us use the definitions of $n$, $m$, $J_0$, $J_1$, $J$, $\seq{\epsilon}$, $P_J$ and $f_j$ introduced in the proof of Theorem~\ref{pro:wav2}. Of course, we have
 \begin{align}
  \lefteqn{2^{J d/p} \| f- P_J \|_{\Lp^p(B(x_0,2^{-J}))}} & \nonumber\\
  &\le \sum_{j=-1}^J 2^{J d/p} \| f_j - \sum_{|\alpha|\le n} \frac{(\cdot -x_0)^\alpha}{|\alpha|!} \, D^\alpha f_j(x_0) \|_{\Lp^p(B(x_0,2^{-J}))} \label{eq:wav s 1}\\
  &\quad + \sum_{j= J+1}^\infty 2^{J d/p} \| f_j \|_{\Lp^p(B(x_0,2^{-J}))}. \label{eq:wav s 2}
 \end{align}

 Let us first consider the term~(\ref{eq:wav s 1}) of the last bound. Let $\alpha$ be a multi-index such that $|\alpha|= n+1$; from Taylor's formula, we need to bound $|D^\alpha f_j(x)|$ for $x\in B(x_0,2^{-J})$. Assume now that $j$ is such that $j-\ce{j/2}\ge J_0$ and define
 \[
 \Lambda_{j,0} := \{ \lambda_{j,k}^{(i)} \in\Lambda_j : |2^j x_0 -k| \le 1\},
 \]
 for $l$ such that $1\le l\le \ce{j/2}$,
 \[
  \Lambda_{j,l} := \{ \lambda_{j,k}^{(i)} \in\Lambda_j : 2^{l-1} < |2^j x_0 - k|\le 2^l \}
 \]
 and
 \[
  \Lambda_{j,*} := \{ \lambda_{j,k}^{(i)} \in\Lambda_j :  |2^j x_0 - k| \ge 2^{\ce{j/2}} \}.
 \]
 A sum over $\Lambda_j$ can be decomposed into a sum over the sets $\Lambda_{j,l}$ (with $l \in \{0,\ldots, \ce{j/2}\})$ and $ \Lambda_{j,*}$. For $1\le l\le \ce{j/2}$, we have, by H\"older's inequality,
 \begin{align*}
  \lefteqn{\sum_{\lambda\in \Lambda_{j,l}} |c_\lambda| \, |D^\alpha \psi_\lambda (x)|} & \\
  &\le (\sum_{\lambda\in \Lambda_{j,l}} |c_\lambda|^p)^{1/p} (\sum_{\lambda\in \Lambda_{j,l}} |D^\alpha \psi_\lambda(y)|^{p'})^{1/p'} \\
  &\le C (\epsilon_{j-l-J_0} 2^{l d/p} \sigma_{j-l}^{-1}) (\sum_{\lambda\in \Lambda_{j,l}} (\frac{1}{(1+| 2^j x-k|)^{2^{d+1}+u+d/p})^{p'}})^{1/p'} \\
  &\le C \epsilon_{j-l-J_0} 2^{-u l} \sigma_{j-l}^{-1},
 \end{align*}
 where $u$ is such that $u>\os(\seq{\sigma})$ and $p'$ is the conjugate exponent of $p$; for $l=0$, we can write
 \[
  \sum_{\lambda\in \Lambda_{j,0}} |c_\lambda| \, |D^\alpha \psi_\lambda (x)| \le \epsilon_{j-J-J_0} \sigma_j^{-1}.
 \]
 For the last set, we get
 \[
  \sum_{\lambda\in \Lambda_{j,*}} |c_\lambda| \, |D^\alpha \psi_\lambda (x)| \le \delta_j \sigma_j^{-1},
 \]
 for a sequence $\seq{\delta}\in \lp^q$, as $f\in B^\eta_{p,q}$. Using these results, we obtain
 \[
 \sum_{\lambda\in \Lambda_j} |c_\lambda| \, |D^\alpha \psi_\lambda (x)| \le
 \delta_j \sigma_j^{-1} + \sum_{l=0}^{\ce{j/2}} \epsilon_{j-l-J0} 2^{-u l} \sigma_{j-l}^{-1}
 \le (\delta_j+ \xi_j) \sigma_j^{-1},
 \]
 where $\seq{\xi} \in \lp^q$ is defined as in the proof of Lemma~\ref{lem:reste}. For the first term~(\ref{eq:wav s 1}), we still have to consider the case $j-\ce{j/2}<J_0$; since $f\in B^\eta_{p,q}$, we can write
 \[
  \sum_{\lambda\in \Lambda_j} |c_\lambda| \, |D^\alpha \psi_\lambda (x)| \le \delta_j 2^{-\eta j} 2^{j d/p} \le C \delta_j \sigma_j^{-1},
 \]
 so that $|D^\alpha f_j(x)|$ is bounded by $C' (\delta_j + \xi_j) 2^{n+1} \sigma_j^{-1}$, for any $j\le J$; we thus have
 \begin{align*}
  \lefteqn{\| f_j - \sum_{|\alpha|\le n} \frac{(\cdot -x_0)^\alpha}{|\alpha|!} \, D^\alpha f_j(x_0) \|_{\Lp^p(B(x_0,2^{-J}))}} & \\
  &\le C (\delta_j+\xi_j) 2^{-(n+1+d/p)J} 2^{(n+1)j} \sigma_j^{-1}.  
 \end{align*}
 Finally, as $\os(\seq{\sigma})<n+1$, (\ref{eq:wav s 1}) is bounded by
 \[
  C 2^{-(n+1)J} \sum_{j=-1}^J (\delta_j+\xi_j) 2^{(n+1)j} \sigma_j^{-1} \le \theta_J \sigma_J^{-1},
 \]
 where $\seq{\theta}\in \lp^q$ is given by Lemma~\ref{lem:reste2}.

 Let us now consider the second term~(\ref{eq:wav s 2}); we actually need to bound the $\Lp^p$-norm of $f_j$ for $j\ge J$. Let us, in the same spirit as before, define
 \[
  \Lambda'_{j,0} := \{\lambda_{j,k}^{(i)} \in \Lambda_j : |2^j x_0 -k| \le 2^{j+J_0-J} \},
 \]
 for $l$ such that $1\le l\le J$,
 \[
  \Lambda'_{j,l} := \{\lambda_{j,k}^{(i)} \in \Lambda_j : 2^{j+J_0-J+l-1}<|2^j x_0 -k| \le 2^{j+J_0-J+l} \}
 \]
 and
 \[
  \Lambda'_{j,*} := \{\lambda_{j,k}^{(i)} \in \Lambda_j : 2^j <|2^j x_0 -k| \}.
 \]
 Using the wavelet characterization of the $\Lp^p$ spaces, we can write
 \begin{align*}
  \| \sum_{\lambda\in \Lambda'_{j,0}} c_\lambda \psi_\lambda \|_{\Lp^p(B(x_0,2^{-J}))}
  &\le C (\sum_{\lambda\in \Lambda_j, \lambda\subset \lambda_J(x_0)} 2^{-dj} |c_\lambda|^p)^{1/p} \\
  &\le C \epsilon_J 2^{-J d/p} \sigma_J^{-1}.
 \end{align*}
 For $l\in \{1,\ldots, J\}$, we get this time
 \[
 \sum_{\lambda\in \Lambda'_{j,l}} |c_\lambda|\, |\psi_\lambda(x)| \le C \epsilon_{J-l} 2^{-(j-J+l)u} \sigma_{J-l}^{-1}
 \le C 2^{-ul} \epsilon_{J-l} \overline{\sigma}_l \sigma_J^{-1},
 \]
 for $x\in B(x_0,2^{-J})$ and
 \[
 \sum_{\lambda\in \Lambda'_{j,*}} |c_\lambda|\, |\psi_\lambda(x)| \le C 2^{-\delta J} \sigma_J^{-1},
 \]
 for some $\delta>0$. As previously, we get that there exists a sequence $\seq{\rho} \in \lp^q$ such that
 \[
  2^{J d/p} \|f_j\|_{\Lp^p(B(x_0,2^{-J}))} \le \rho_J \sigma_J^{-1},
 \]
 so that we can conclude using the same arguments as in the compactly supported case.
\end{proof}

\section{A multifractal formalism}
We show here that the generalized Besov spaces provide a natural framework for the multifractal formalism based on the $\ep^{\seq{\sigma}}_{p,q}$ spaces.
\subsection{Definitions}
As the wavelet leaders method rests on the oscillation spaces $\Os^{s,s'}_{p}$ (see \cite{Jaffard:04b,Jaffard:05b}), we need to adapt these spaces to our general framework.
\begin{Def}
 Let $p,q,r\in [1,\infty]$; a function $f$ belongs to $\Os^{\seq{\sigma}}_{p,r,q}$ if the sequence $(C_k)_k$ defined by (\ref{eq:def Ck}) belongs to $\lp^q$ and if
 \[
  (\sum_{j\in\N} (\sum_{\lambda\in \Lambda_j} (\sigma_j 2^{-dj/r} d_\lambda^p)^r)^{q/r})^{1/q} \le C,
 \]
 for some constant $C>0$.
\end{Def}
We will show that these spaces are closely related to the generalized Besov spaces introduced in Definition~\ref{def:gen bes}; as expected, these are obtained by replacing the dyadic sequence appearing in the usual definition with an admissible sequence (see e.g.\ \cite{Loosveldt:19}).
\begin{Def}\label{def:gen bes}
 Let $(\phi_j)_j$ be a sequence of functions belonging to the Schwartz class $\cs$ such that
 \begin{itemize}
  \item $\supp(\phi_0) \subset \{x\in \Rd: |x|\le 2 \}$,
  \item $\supp(\phi_j) \subset \{x\in \Rd: 2^{j-1}\le |x|\le 2^{j+1} \}$ ($j\ge 1$),
  \item for any multi-index $\alpha$, there exists a constant $C_\alpha>0$ such that we have $\displaystyle \sup_{x\in \Rd} |D^\alpha \phi_j(x)| \le C_\alpha 2^{-j|\alpha|}$,
  \item $\displaystyle \sum_{j\in \N} \phi_j =1$.
 \end{itemize}
 Given $p,q \in [1,\infty]$; the space $B^{\seq{\sigma}}_{p,q}$ is defined as
\[
 B^{\seq{\sigma}}_{p,q} := \{ f\in \cs' : (\sum_{j\in\N} \|\sigma_j \mathcal{F}^{-1}(\phi_j \mathcal{F} f) \|_{\Lp^p}^q)^{1/q} < \infty \},
\]
where $\ft$ denotes the Fourier transform.
\end{Def}
We will use the wavelet characterization of these spaces (see \cite{Almeida:05}).
\begin{Def}
 Let $p,q\in [1,\infty]$; if, given $h>-d/p$, $\seq{\gamma}^{(h)}$ is an admissible sequence, the family of admissible sequences $h\mapsto \seq{\gamma}^{(h)}$ is $(p,q)$-decreasing if it satisfies $\us(\seq{\gamma}^{(h)}) >-d/p$, $\underline{\gamma}_1^{(h)} > 2^{-d/p}$ for any $h>-d/p$ and if $-d/p<h<h'$ implies
\[
 \ep^{\seq{\gamma}^{(h)}}_{p,q} (x_0) \subset \ep^{\seq{\gamma}^{(h')}}_{p,q} (x_0).
\]
\end{Def}
In the sequel, we will only consider families of admissible sequences $\seq{\gamma}^{(\cdot)}$ that are implicitly $(p,q)$-decreasing. This notion was introduced in \cite{Kreit:18}, where criteria to obtain such families were presented.

A multifractal formalism is an empirical method that allows to estimate the quantity
\[
 \ha \{x_0\in \Rd: h(x_0)= h\},
\]
where $\ha$ denotes the Hausdorff dimension with the convention $\ha(\emptyset)=-\infty$ (see \cite{Falconer:03} for example) and $h(x_0)$ is some kind of generalized H\"older exponent. Usually, one requires such a method to be valid for a large class of functions. We aim at providing here a multifractal formalism for the exponents (\ref{eq:def h}) defined from the $T^{\seq{\sigma}}_{p,q}$ spaces, thus generalizing the wavelet leaders method \cite{Jaffard:04b,Jaffard:05b}.
\begin{Def}
 Given $p,q\in [1,\infty]$ and a family of admissible sequences $\seq{\gamma}^{(\cdot)}$, the generalized $(p,q)$-H\"older exponent associated to $f\in \Lpl^p$ and $\seq{\gamma}^{(\cdot)}$ at $x_0\in \Rd$ is defined by
 \begin{equation}\label{eq:def h}
  h_{p,q}(x_0) := \sup \{h>-d/p: f\in \ep^{\seq{\gamma}^{(h)}}_{p,q}(x_0) \}.
 \end{equation}
\end{Def}
The most natural family of admissible sequences is $h\mapsto (2^{jh})_j$; in this case, $h_{\infty,\infty}(x_0)$ is the usual H\"older exponent \cite{Jaffard:04b}, while $h_{p,\infty}(x_0)$ is the $p$-exponent considered in \cite{Jaffard:05b}.

Given $p,q\in [1,\infty]$, a family of admissible sequences $\seq{\gamma}^{(\cdot)}$ and a function $f\in\Lpl^p$, we set
\[
 \ms_{p,q}(h) := \ha \{x_0\in \Rd: h_{p,q}(x_0)= h\}.
\]

\subsection{Preliminary results}
In this section, we will implicitly work with indices $p,q,r\in [1,\infty]$, a function $f$ that belongs to $\Lpl^p$, a point $x_0\in \Rd$, a family of admissible sequences $\seq{\gamma}^{(\cdot)}$ and an admissible sequence $\seq{\sigma}$.
\begin{Lemma}
 If \[
 \gamma^{(h)}_j 2^{\eta j} d_j^p(x_0) \in \lp^q,
\]
for some $\eta >0$ such that $\fl{\os(\seq{\gamma}^{(h)})+ \eta} = \fl{\os(\seq{\gamma}^{(h)})}$, then $h_{p,q}(x_0) \ge h$.
\end{Lemma}
\begin{proof}
 We know that there exists a sequence of polynomial $(P_j)_j$ of degree at most $\os(\seq{\gamma}^{(h)})$ and a sequence $\seq{\epsilon} \in \lp^q$ such that
\[
 \gamma_j^{(h)} 2^{j d/p} \| f - P_j \|_{\Lp^p(B(x_0,2^{-j}))} \le C \epsilon_j 2^{-\eta j} |\log_2(2^{-\eta j -j d/p}/\gamma_j^{(h)})|,
\]
for $j$ large enough, which implies $f\in \ep^{\seq{\gamma}^{(h)}}_{p,q}(x_0)$.
\end{proof}
\begin{Prop}
 If the function $f$ belongs to both $B^\eta_{p,q}$ for some $\eta>0$ and $\Os^{\seq{\sigma}}_{p,r,q}$, then
 \[
  \ha(x_0\in \Rd: h_{p,q}(x_0) <h \} \le d+ r \os(\frac{\seq{\gamma}^{(h)}}{\seq{\sigma}}).
 \]
\end{Prop}
\begin{proof}
 Let $\seq{\epsilon} \in \lp^q$ be such that $\epsilon_j\not=0$ and
 \[
  (\sum_{\lambda\in \Lambda_j} (\sigma_j 2^{-jd/r} d_\lambda^p)^r)^{1/r} \le \epsilon_j,
 \]
for all $j\in\N$. Let us first consider the case $r=\infty$; if $\os(\seq{\gamma}^{(h)}/\seq{\sigma})<0$, there exists $\delta>0$ such that $\gamma_j^{(h)} 2^{\delta j} d_j^p(x_0)\le C \epsilon_j$ for any $j$ and $h_{p,q}(x_0)\ge h$ for all $x_0\in\Rd$. As a consequence, we have
\[
 \ha\{ x_0\in \Rd : h_{p,q}(x_0)<h \} =-\infty = d +r\os(\seq{\gamma}^{(h)}/\seq{\sigma}).
\]
On the other hand, if $\os(\seq{\gamma}^{(h)}/\seq{\sigma})\ge 0$,
\[
 \ha\{ x_0\in \Rd : h_{p,q}(x_0)<h \} \le d \le d +r\os(\seq{\gamma}^{(h)}/\seq{\sigma}).
\]

Now, suppose $r<\infty$, fix $h>-d/p$ and define, given $j\in\N$ and $\delta>0$ sufficiently small,
\[
 E_{j,\delta}^h := \{ \lambda\in \Lambda_j : d_\lambda^p \ge \epsilon_j 2^{-\delta j}/\gamma_j^{(h)} \}
\]
and set $n= \# E_{j,\delta}^h$. As $f\in \Os^{\seq{\sigma}}_{p,r,q}$, we have
\[
 \sigma_j^r 2^{-jd} n (2^{-\delta j}/\gamma_j^{(h)})^r
 \le \epsilon_j^{-r} \sigma_j^r 2^{-jd} \sum_{\lambda\in E_{j,\delta}^h} (d_\lambda^p)^r \le 1,
\]
so that
\[
 n\le 2^{jd} (2^{-\delta j}/ \gamma_j^{(h)})^{-r}/\sigma_j^r.
\]
Now, define $\Lambda_{j,\delta}^h$ as the set of the dyadic cubes $\lambda\in \Lambda_j$ such that there exists a neighbor $\lambda'\in 3\lambda$ that belongs to $E_{j,\delta}^h$. Finally, define
\[
 F_\delta^h := \limsup_j \{x_0\in \Rd: \lambda_j(x_0) \in \Lambda_{j,\delta}^h\}.
\]
If $x_0$ does not belong to $F_\delta^h$, then there exists $J\in\N$ such that $j\ge J$ implies $\lambda_j(x_0)\not\in \Lambda_{j,\delta}^h$ and, from what we have obtained for $n$, there exists a constant $C>0$ for which $j\ge J$ implies
\[
 2^{\delta j} \gamma_j^{(h)} d_j^p(x_0) \le C \epsilon_j,
\]
and therefore
\begin{equation}\label{eq:incl fdh}
 \{x_0\in \Rd: h_{p,q}(x_0)< h\} \subset F_\delta^h. 
\end{equation}

Let $\alpha>0$, set $j_1 := \inf \{ j: \sqrt{d}2^{-j}<\alpha\}$ and
\[
 E_\delta := \{ \lambda\in \Lambda_{j,\delta}^h : j\ge j_1\}.
\]
It is easy to check that $E_\delta$ is an $\alpha$-covering of $F_\delta^h$; given $s,\eta>0$, we have
\begin{align*}
 \sum_{\lambda\in E_\delta} \text{diam}(\lambda)^s
 &\le \sum_{j\ge j_0} \# F_j^h (\sqrt{d} 2^{-j})^s \\
 &\le C \sum_{j\ge j_0} 2^{(d-s)j} (2^{-\delta j}/\gamma_j^{(h)})^{-r}/\sigma_j^r \\
 &\le C' \sum_{j\in\N} 2^{rj(\os(\seq{\gamma}^{(h)}/\seq{\sigma})+\delta+\eta} 2^{(d-s)j}.
\end{align*}
As a consequence, we have
\[
 \ha (F_\delta^h) \le d +  r(\os(\frac{\seq{\gamma}^{(h)}}{\seq{\sigma}})+\delta+\eta),
\]
for any $\eta>0$ and we can conclude thanks to (\ref{eq:incl fdh}).
\end{proof}
Of course, for the natural choices of the families of admissible sequences \cite{Kreit:18}, $h\mapsto \os(\seq{\gamma}^{(h)})$ is continuous; in such a case, the previous result can be improved.
\begin{Rmk}
 If there exists a sequence $\seq{\epsilon}$ converging to $0^+$ such that
 \[
  \os(\frac{\seq{\gamma}^{(h+\epsilon_j)}}{\seq{\sigma}}) \to \os(\frac{\seq{\gamma}^{(h)}}{\seq{\sigma}}),
 \]
 we have
 \[
  \ha\{ x_0\in \Rd: h_{p,q}(x_0)\le h\} \le d+ r\os(\frac{\seq{\gamma}^{(h)}}{\seq{\sigma}}).
 \]
\end{Rmk}

\begin{Prop}
 If $\seq{\sigma}$ is an admissible sequence such that $\us(\seq{\sigma})>0$ and $\us(\seq{\sigma}) -d/r> -d/p$, we have $\Os^{\seq{\sigma}}_{p,r,q}= B^{\seq{\sigma}}_{r,q}$.
\end{Prop}
\begin{proof}
 We obviously have $\Os^{\seq{\sigma}}_{p,r,q}\hookrightarrow B^{\seq{\sigma}}_{r,q}$. If $f$ belongs to $B^{\seq{\sigma}}_{r,q}$, we have
 \begin{align}
  \lefteqn{(\sum_{\lambda\in \Lambda_j} (\sigma_j 2^{-j d/r} d_\lambda^p)^r)^{q/r}} & \nonumber \\
  &\le \Big(\sum_{\lambda\in \Lambda_j} (\sigma_j 2^{-jd/r})^q  \sum_{j'\ge j} \big(\sum_{\lambda'\in \Lambda_{j'}, \lambda'\subset \lambda} (2^{(j-j')d/p} |c_{\lambda'}|)^p \big)^{r/p} \Big)^{q/r}, \label{eq:os=bes}
 \end{align}
 for any $j\in\N$.

 Let us first suppose that $r\le p$; in this case, (\ref{eq:os=bes}) is bounded by
 \[
  (\sum_{j'\ge j} (\sigma_j \sigma_{j'}^{-1} 2^{(j-j')d/p} 2^{(j'-j)d/r})^r \sum_{\lambda'\in \Lambda_{j'}} (\sigma_{j'} 2^{-j' d/r} |c_{\lambda'}|)^r )^{q/r}.
 \]
 Let $\epsilon>0$ be such that $\us(\seq{\sigma}) -\epsilon -d/r > -d/p$; there exists a constant $C_\epsilon>0$ such that
 \[
  \sigma_j \sigma_{j'}^{-1} < C_\epsilon 2^{(\us(\seq{\sigma})-\epsilon)(j-j')}.
 \]
 If $q\le r$, (\ref{eq:os=bes}) is bounded by
 \[
  C ( \sum_{j'\ge j} (2^{(\us(\seq{\sigma})+\epsilon-d/r-d/r)(j-j')})^q (\sum_{\lambda'\in \Lambda_{j'}} (\sigma_{j'} 2^{-j' d/r} |c_{\lambda'}|)^r )^{q/r} ).
 \]
 As $f$ belongs to $B^{\seq{\sigma}}_{r,q}$, we can write
 \[
  (\sum_{j\in \N}(\sum_{\lambda\in \Lambda_j} (\sigma_j 2^{-j d/r} d_\lambda^p)^r)^{q/r})^{1/q}
  \le C (\sum_{j'\in\N} (\sum_{\lambda'\in \Lambda_{j'}} (\sigma_{j'} 2^{-j' d/r} |c_{\lambda'}|)^r)^{q/r})^{1/q},
 \]
 wich implies $f\in \Os^{\seq{\sigma}}_{p,r,q}$. If $r<q$, by denoting $s$ the conjugate exponent of $q/r$, we can use H\"older's inequality to bound (\ref{eq:os=bes}) by
 \begin{align*}
  \lefteqn{C (\sum_{j'\ge j} (2^{-\us(\seq{\sigma}) + \epsilon -d/p -d/r)(j'-j)})^{rs/2})^{q/(rs)}} & \\
  &\quad \Big( \sum_{j'\ge j} (2^{-\us(\seq{\sigma}) + \epsilon -d/p -d/r)(j'-j)})^{q/(2r)} (\sum_{\lambda'\in \Lambda_{j'}} (\sigma_{j'} 2^{-j' d/r} |c_{\lambda'}|)^r)^{q/r} \Big) \\
  & \le C \Big( \sum_{j'\ge j} (2^{-\us(\seq{\sigma}) + \epsilon -d/p -d/r)(j'-j)})^{q/(2r)} (\sum_{\lambda'\in \Lambda_{j'}} (\sigma_{j'} 2^{-j' d/r} |c_{\lambda'}|)^r)^{q/r} \Big),
 \end{align*}
 so that $f$ belongs to $\Os^{\seq{\sigma}}_{p,r,q}$, as in the previous case.

 We still have to consider the case $p<r$; by Jensen's inequality, we can bound (\ref{eq:os=bes}) by
 \begin{align*}
  \lefteqn{(\sum_{\lambda\in \Lambda_j} (\sigma_j 2^{-j d/r} )^r \sum_{j\ge j'} \sum_{\lambda'\in \Lambda_{j'}, \lambda'\subset \lambda} 2^{(j-j')d} |c_{\lambda'}|^r)^{q/r}} & \\
  &\le (\sum_{j'\ge j} (\sigma_j /\sigma_{j'})^r \sum_{\lambda'\in \Lambda_{j'}} (\sigma_{j'} 2^{-j' d/r} |c_{\lambda'}|)^r)^{q/r},
 \end{align*}
 so that we can conclude as in the other cases.
\end{proof}

\subsection{The notion of prevalence}
In this section, we very briefly introduce the notion of prevalence (see \cite{Christensen:72,Hunt:92,Hunt:94} for more details).

In $\Rd$, it is well known that if one can associate a probability measure $\mu$ to a Borel set $B$ such that $\mu(B+x)$ vanishes for very $x\in\Rd$, then the Lebesgue measure $\leb(B)$ of $B$ also vanishes. For the notion of prevalence, this property is turned into a definition in the context of infinite-dimensional spaces.
\begin{Def}\label{def:prev}
 Let $E$ be a complete metric vector space; a borel set $B$ of $E$ is Haar-null if there exists a compactly-supported probability measure $\mu$ such that $\mu(B+x)=0$, for every $x\in E$. A subset of $E$ is Haar-null if it is contained in a Haar-null Borel set; the complement of a Haar-null set is a prevalent set.
\end{Def}
Is $E$ is finite-dimensional, $B$ is Haar-null if and only if $\leb(B)=0$; if $E$ is infinite-dimensional, the compact sets of $E$ are Haar-null. Moreover, it can be shown that a translated of a Haar-null set is Haar-null and that a prevalent set is dense in $E$. Finally, the intersection of a countable collection of prevalent sets is prevalent.

Let us make some remarks about how to show that a set is Haar-null. A common choice for the measure in Definition~\ref{def:prev} is the Lebesgue measure on the unit ball of a finite-dimensional subset $E'$ of $E$. For such a choice, one has to show that $\leb(B\cap (E'+x))$ vanishes for every $x$. If $E$ is a function space, one can choose a random process $X$ whose sample paths almost surely belong to $E$. In this case, one can show that a property only holds on a Haar-null set by showing that the sample path $X$ is such that, for any $f\in E$, $X_t+f$ almost surely does not satisfy the property.

If a property holds on a prevalent set, we will say that it holds almost everywhere from the prevalence point of view.

\subsection{A multifractal formalism associated to the generalized Besov spaces}
We propose here the following formula to estimate the spectrum $\ms_{p,q}$ related to a function $f\in B^{\seq{\sigma}}_{r,s}$:
\[
 \ms_{p,q}(h) = d+ r \os(\frac{\seq{\gamma}^{(h)}}{\seq{\sigma}})
\]
and show that, under natural smooth conditions, this equality is satisfied almost everywhere from a prevalent point of view.

\begin{Def}\label{def:zeta}
 An admissible sequence $\seq{\sigma}$ and a family of admissible sequences $\seq{\gamma}^{(\cdot)}$ are compatible for $p,q,r,s\in [1,\infty]$ with $s\le q$ if
 \begin{itemize}
  \item $\us(\seq{\sigma})>0$,
  \item $\us(\seq{\sigma})-d/r> -d/p$
  \item $\zeta$, defined on $(-d/p,\infty)$ by
 \[
  \zeta(h):= \us(\frac{\seq{\gamma}^{(h)}}{\seq{\sigma}})= \os(\frac{\seq{\gamma}^{(h)}}{\seq{\sigma}}),
 \]
 is a non decreasing continuous function such that
 \[
 \{h>-d/p: \zeta(h)<-d/r\} \not=\emptyset.
 \]
 \end{itemize}
 We call $\zeta$ the ratio function. We will also frequently use the quantity
 \[
  h_{\min}(r) := \sup\{h> -d/p: \zeta(h)< -d/r \}.
 \]
\end{Def}
The following remark stresses the importance of $h_{\min}$.
\begin{Rmk}
 Suppose that $\seq{\sigma}$ and $\seq{\gamma}^{(\cdot)}$ are compatible as in the previous definition. If $f$ belongs to $B^{\seq{\sigma}}_{p,q}$, there exists $\eta>0$ such that $B^{\seq{\sigma}}_{p,q} \hookrightarrow B^\eta_{p,q}$. For $\lambda\in \Lambda_j$ and $j'\ge j$, we have
 \begin{align*}
  (\sum_{\lambda'\in \Lambda_{j'}, \lambda'\subset \lambda} (2^{(j-j')d/p} |c_{\lambda'}|)^p )^{1/p}
  &\le 2^{j d/p} (\sum_{\lambda'\in \Lambda_{j'}} (\sigma_{j'} 2^{-j' d/p} |c_{\lambda'}|)^p )^{1/p} \sigma_{j'}^{-1} \\
  &\le 2^{j d/p}\epsilon_{j'} \sigma_{j'}^{-1},
 \end{align*}
 for a sequence $\seq{\epsilon}\in \lp^q$. As a consequence, there exists $\eta>0$ and a sequence $\seq{\xi}\in \lp^q$ given by Lemma~\ref{lem:reste} such that, for $\lambda\in\Lambda_j$,
 \[
  d_\lambda^p \le C \sum_{j'\ge j} 2^{j d/p} \epsilon_{j'} \sigma_{j'}^{-1}
  \le C \xi_j 2^{-\eta j} /\gamma_j^{(h)}
 \]
 for all $h>-d/p$ such that $\os(\seq{\gamma}^{(h)}/\seq{\sigma})<- d/p$. As a consequence, one has $h_{p,q}(x_0) \ge h_{\min}(p)$, for any $x_0\in \Rd$.

 In the same spirit, for $r\le p$, one has $B^{\seq{\sigma}}_{r,q} \hookrightarrow B^{\seq{\theta}}_{p,q}$, where $\seq{\theta}$ is the admissible sequence defined by $\theta_j := 2^{(d/p-d/r)j} \sigma_j$ ($j\in\N$). As $\us(\seq{\sigma}) -d/r> -d/p$ implies $\us(\seq{\theta})>0$, there exists $\eta>0$ such that $B^{\seq{\sigma}}_{r,q} \hookrightarrow B^{\eta}_{p,q}$ and $h_{p,q}(x_0) \ge h_{\min}(r)$, for any $x_0\in \Rd$.

 That being done, if $p<r$ then, for any $f\in B^{\seq{\sigma}}_{r,q}$,
 \[
  h_{p,q}(x_0) \ge h_{r,q}(x_0) \ge h_{\min}(r).
 \]
 Thus, if $f\in B^{\seq{\sigma}}_{r,s}$, we have $f\in B^{\seq{\sigma}}_{r,q}$ and $h_{p,q}(x_0) \ge h_{\min}(r)$.
\end{Rmk}
From what we have done so far, we get the following corollary.
\begin{Cor}
 Let $p,q,r,s\in [1,\infty]$, $\seq{\sigma}$ be an admissible sequence and $\seq{\gamma}^{(\cdot)}$ be a family of admissible sequences such that $\seq{\sigma}$ and $\seq{\gamma}^{(\cdot)}$ are compatible. If $f$ belongs to $B^{\seq{\sigma}}_{r,s}$, then
 \begin{itemize}
  \item $\{x\in \Rd: h_{p,q}(x_0)\le h\}= \emptyset$, \quad for any $h<h_{\min}(r)$,
  \item $\ha(\{x\in \Rd: h_{p,q}(x_0)\le h\})\le d+ r g(h)$, \quad for any $h\ge h_{\min}(r)$.
 \end{itemize}
\end{Cor}

To show that, under some general hypothesis, the last upper bound is optimal for a prevalent set of functions in $B^{\seq{\sigma}}_{r,s}$, we need the following definition.
\begin{Def}
 Let $x_0\in \Rd$ and $r>0$; the strict cone of influence above $x_0$ of width $r$ is
 \[
  \mathcal{C}_{x_0}(r) := \{(j,k) \in \N\times \Z^d: \| \frac{k}{2^j} -x_0\|_{\infty} < \frac{r}{2^j} \},
 \]
 where $\|x-x_0\|_{\infty}$ is the Chebyshev distance between $x$ and $x_0$.
\end{Def}
This definitions is related to the wavelets as follows: in this context, we set
\[
 \K_{x_0}(r) := \{\lambda_{j,k}^{(i)}\in \Lambda : (j,k)\in \mathcal{C}_{x_0}(r)\}.
\]
The following result explains why $\K_{x_0}$ can be seen as a cone of influence for the wavelets.
\begin{Prop}
 If $f$ belongs to $\ep^{\seq{\sigma}}_{p,q}(x_0)$, then
 \[
 ( \sigma_j \sum_{\lambda\in \Lambda_j\cap \K_{x_0}(r)} |c_\lambda|^p)^{1/p})_j \in \lp^q.
 \]
\end{Prop}
\begin{proof}
 Choose $j_1\in \N$ such that $\sqrt{d} r + 2^{j_0}\le 2^{j_1}$; for $j\ge j_1$, if $\lambda \in \Lambda_j$ also belongs to $\K_{x_0}(r)$, then the support of $\psi_\lambda$ is included in $B(x_0,2^{j_1-j})$. From the proof of Theorem~\ref{pro:wav1}, we know that there exists a sequence $\seq{\epsilon}\in \lp^q$ such that
 \[
  \sigma_j (\sum_{\lambda\in \Lambda_j\cap \K_{x_0}(r)} |c_\lambda|^p)^{1/p} \le \epsilon_j,
 \]
 for any $j\ge j_1$. The conclusion then comes from the Archimedean property of the real line.
\end{proof}

Given a dyadic cube $\lambda=\lambda_{j,k}^{(i)}$, let us denote by $k(\lambda)$ and $j(\lambda)$ the numbers such that $k(\lambda)/2^{j(\lambda)}$ is the dyadic irreducible form of $k/2^j$. For $\alpha\in [1,\infty]$, let us set
\[
 h_*(\alpha) := \zeta^{-1} (\frac{d}{\alpha r} - \frac{d}{r}).
\]
We have $h_*(\alpha)\ge h_{\min}(r)= h_*(\infty)$. If $\zeta(h)>d/\alpha r - d/r$, choose $\epsilon_0>0$ such that $\zeta(h)-\epsilon_0> d/\alpha r -d/r$ and let $m_0\in \N$ be such that
\begin{equation}\label{eq:def m0}
 d- (\frac{d}{\alpha r} - \frac{d}{r} -\zeta(h) +\epsilon_0) 2^{d m_0} \alpha <0.
\end{equation}
Let us split each cube $\lambda\in \Lambda_j$ into $2^{d m_0}$ cubes of scale $j+m_0$ and for each $n\in \{1,\ldots, 2^{d m_0}\}$, choose a unique subcube $\lambda^{(n)}$ of $\lambda$ such that $n\not= n'$ implies $\lambda^{(n)} \not= \lambda^{(n')}$. From this, we can consider a function $g^{(n)}$ such that its wavelet coefficients $c_\lambda$  satisfy the following conditions:
\[
 c_{\lambda^{(n)}} := j^{-a_0} 2^{j d/r} 2^{- j(\lambda)d/r} \sigma_j^{-1}
 \qquad \text{if } \lambda\in \Lambda_j\cap [0,1]^d,
\]
with $a_0 := 1+1/r+1/s$ and $c_\lambda:=0$ if $\lambda$ is not of the form $\lambda^{(n)}$ for some $n$.
\begin{Prop}
 For all $n\in \{1,\ldots, 2^{d m_0}\}$, $g^{(n)}$ belongs to $B^{\seq{\sigma}}_{r,s}$.
\end{Prop}
\begin{proof}
 For $j\ge 1$, we have
 \begin{align*}
  \lefteqn{(\sum_{\lambda \in \Lambda_{j+m_0}} (\sigma_{j+m_0} 2^{-(j+m_0)d/r} |c_\lambda| )^r)^{1/r}} & \\
  &= (\sum_{l=0}^j \sum_{\dbline{\lambda\in \Lambda_j\cap [0,1]^d}{j(\lambda)=l}} (\sigma_{j+ m_0} 2^{-(j+m_0)d/r} j^{-a_0} 2^{j d/r} 2^{-l d/r} \sigma_j^{-1})^r)^{1/r}
 \end{align*}
 and
 \begin{align*}
  (\sum_{\lambda\in \Lambda_{j+m_0}} (\sigma_{j+m_0} 2^{-(j+m_0) d/r} |c_\lambda|)^r)^{1/r}
  &\le (\sum_{l=0}^j (\overline{\sigma}_{m_0} 2^{-(j+m_0) d/r} j^{-a_0})^r)^{1/r} \\
  &\le C j^{-a_0+1/r}.
 \end{align*}
 As $a_0>1/r+1/s$, we get
 \[
  (\sum_{j\ge 1} (\sum_{\lambda\in \Lambda_{j+m_0}} (\sigma_{j+m_0} 2^{-(j+m_0)d/r} |c_\lambda|)^r)^{s/r})^{1/s}
  \le C (\sum_{j\ge 1} j^{-s(a_0-1/r)})^{1/s} <\infty,
 \]
 which is sufficient to conclude.
\end{proof}
\begin{Def}
 Let $\alpha\ge 1$; a point $x_0\in [0,1]^d$ is $\alpha$-approximable by dyadics if there exists two sequences $\seq{k}$ and $\seq{j}$ of natural numbers with $k_n< 2^{j_n}$ for any $n\in \N$ such that
 \[
  \| x_0 -\frac{k_n}{j_n}\|_\infty \le \frac{1}{2^{\alpha j_n}},
 \]
 for any $n\in \N$.
\end{Def}
Let us denote the set of points of $[0,1]^d$ which are $\alpha$-approximable by dyadics by $E^\alpha$ and define
\[
 E_j^\alpha := \{ x_0\in [0,1]^d : \exists k\in \{0,\ldots, 2^j -1\}^d \text{ such that } \| x_0 - \frac{k}{2^j}\|_\infty \le \frac{1}{2^{\alpha j}} \},
\]
so that $E^\alpha = \limsup_j E_j^\alpha$. We also define
\[
 E_{j,k}^\alpha := \{ x_0\in [0,1]^d: \| x_0- \frac{k}{2^j}\|_\infty \le \frac{1}{2^{\alpha j}} \},
\]
for $k\in \{0,\ldots, 2^j-1\}^d$, in order to have
\[
 E_j^\alpha = \bigcup_{l\in\{0,\ldots, 2^j-1\}^d} E_{j,k}^\alpha.
\]
Finally, set $E^\infty= \cap_{\alpha\ge 1} E^\alpha$; this set in not empty since it contains the dyadic numbers.
\begin{Prop}
 Given $C>0$, $j\in \N$ and $k\in \{0,\ldots, 2^j-1\}^d$, the set
 \begin{align*}
  \lefteqn{F_{j,k}^{\alpha,C} (h)} & \\
  &:= \{ f\in B^{\seq{\sigma}}_{r,s}: ( \exists x\in E_{j,k}^\alpha : \forall n\in \N \forall \lambda\in \Lambda_n\cap \K_x(2^{m_0+1}), |c_\lambda|\le C/\gamma_n^{(h)}) \}
 \end{align*}
 is closed in $B^{\seq{\sigma}}_{r,s}$.
\end{Prop}
\begin{proof}
 Let $(f_l)_l$ be a sequence of functions of $F_{j,k}^{\alpha,C}$ such that $f_l\to f$ in $B^{\seq{\sigma}}_{r,s}$ and denote by $c_\lambda^{(l)}$ (resp.\ $c_\lambda$) the wavelet coefficients of $f_l$ (resp.\ $f$). Since
 \[
  B^{\os(\seq{\sigma})+\gamma}_{r,s} \hookrightarrow B^{\seq{\sigma}}_{r,s} \hookrightarrow B^{\us(\seq{\sigma})-\gamma}_{r,s},
 \]
 for any $\gamma>0$ and as the application which associates to a function its wavelet coefficients is continuous in the Besov spaces, we have $c_\lambda^{(l)} \to c_\lambda$ for all $\lambda\in \Lambda$.

 For $l\in\N$, let $x_l \in E_{j,k}^\alpha$ be such that, for all $n\in\N$ and $\lambda\in \Lambda_n\cap \K_{x_l}(2^{m_0+1})$, we have $|c_\lambda^{(l)}|\le C/ \gamma_n^{(h)}$. As $E_{j,k}^\alpha$ is compact, we can suppose that the sequence $(x_j)_l$ converges to a point $x_0$ of $E_{j,k}^\alpha$. Now, let us fix $N\in \N$ and $\delta>0$; if $l$ is sufficiently large, we have $\K_{x_0}(2^{m_0+1}) \subset K_{x_l}(2^{m_0+1})$ and, for $n\le N$, we have, for $\lambda\in \Lambda_n\cap \K_{x_l}(2^{m_0+1})$, $|c_\lambda^{(l)}- c_\lambda|\le \delta/\gamma_n^{(h)}$ as $c_\lambda^{(l)}$ converges to $c_\lambda$. Also, we have $|c_\lambda^{(l)}|\le C/\gamma_n^{(h)}$ for $\lambda\in \Lambda_n\cap \K_{x_l}(2^{m_0+1})$. As a consequence, $\lambda\in \Lambda_n\cap \K_{x_0}(2^{m_0+1})$ implies
 \[
  |c_\lambda|\le (C+\delta)/\gamma_n^{(h)},
 \]
 for all $n\le N$. Taking the limit for $N\to \infty$ and $\delta\to 0^+$ leads to $f\in F_{j,k}^{\alpha,C}(h)$.
\end{proof}
Let us set
\[
 F_j^{\alpha,C}(h) := \bigcup_{k\in\{0,\ldots, 2^j-1\}^d} F_{j,k}^{\alpha,C}(h)
\]
and $F^{\alpha,C}(h):= \limsup_j F_j^{\alpha,C} (h)$.
All these sets are obviously Borel sets.
\begin{Prop}
 The set $F^{\alpha,C}(h)$ is a Haar-null Borel set.
\end{Prop}
\begin{proof}
 Set $m_1 := 2^{m_0 d}$ and let us fix $j\in\N$ and $k\in \{0,\ldots, 2^j -1\}$; for $f\in B^{\seq{\sigma}}_{r,s}$, suppose that there exist two points of $\R^{m_1}$, $a^{(1)}=(a_1^{(1)},\ldots,a_{m_1}^{(1)})$ and $a^{(2)}=(a_1^{(2)},\ldots,a_{m_1}^{(2)})$, such that
 \[
  f_l := f+ \sum_{m=1}^{m_1} a_m^{(l)} g^{(m)}
 \]
 belongs to $F_{j,k}^{\alpha,C}(h)$ $(l\in \{1,2\})$. For $l\in\{1,2\}$, let us also denote by $c_\lambda^{(l)}$ the wavelet coefficient of $f_l$ associated to the dyadic cube $\lambda\in \Lambda$ and let $x_l$ be a point of $E_{j,k}^\alpha$ such that $\lambda \in \Lambda_{\fl{\alpha j}}\cap \K_{x_l}(2^{m_0+1})$ implies $|c^{(l)}_\lambda| \le C/\gamma_{\fl{\alpha j}}^{(h)}$.

 For $\lambda'\in \Lambda_{\fl{\alpha j}+m_0}$ satisfying $\lambda'\subset \lambda_{\fl{\alpha j},k}^{(i)}$, we have
 \[
  |c_{\lambda'}^{(l)}| \le C/ \gamma_{\fl{\alpha j}+m_0}^{(h)}.
 \]
 As a consequence, we get, by denoting $c'^{(m)}_\lambda$ the wavelet coefficient of $g^{(m)}$ associated to $\lambda$,
 \[
  |a^{(1)}_m - a^{(2)}_m|
  = |a^{(1)}_m - a^{(2)}_m|\, | c'^{(m)}_{\lambda^{(m)}} |/ |c'^{(m)}_{\lambda^{(m)}} |
  \le 2C/ (\gamma_{\fl{\alpha j}+m_0}^{(h)} | c'^{(m)}_{\lambda^{(m)}} |),
 \]
 for any $m\in\{1,\ldots, m_1\}$. On the other hand, for $j\ge j(\lambda)$, we have
 \begin{align*}
  |c'^{(m)}_{\lambda^{(n)}}|
  &= \fl{\alpha j}^{-a_0} 2^{\fl{\alpha j} d/q} 2^{- j(\lambda) d/q} \sigma_{\fl{\alpha j}}^{-1} \\
  &\ge C' \fl{\alpha j}^{-a_0} 2^{\fl{\alpha j} d/q} 2^{-\fl{\alpha j} d/\alpha q} \sigma_{\fl{\alpha j}}^{-1},
 \end{align*}
 so that there exists a constant $C''>0$ for which
 \begin{equation}\label{eq:HNB}
  \| a^{(1)}- a^{(2)}\|_\infty
  \le C'' \fl{\alpha j}^{-a_0} 2^{\fl{\alpha j}(d/\alpha q- d/q)} \sigma_{\fl{\alpha j}} /\gamma_{\fl{\alpha j}}^{(h)}.
 \end{equation}

 That being done, for $f\in B^{\seq{\sigma}}_{r,s}$, we have
 \begin{align*}
  \{a \in \R^{m_1} : f+ a g \in F^{\alpha,C}(h)\}
  &\subset \bigcup_{j\ge J} \{a \in \R^{m_1} : f+ a g \in F_j^{\alpha,C}(h)\} \\
  &\subset \bigcup_{j\ge J} \bigcup_{k\in \{0,\ldots, 2^j-1\}^d} \{a \in \R^{m_1} : f+ a g \in F_{j,k}^{\alpha,C}(h)\},
 \end{align*}
 for any $J\in \N$. Thus, from (\ref{eq:HNB}), we get
 \begin{align*}
  \lefteqn{\leb (\{a \in \R^{m_1} : f+ a g \in F^{\alpha,C}(h)\})} & \\
  &\le \sum_{j\ge J} 2^{jd} (C'' \fl{\alpha j}^{a_0} 2^{\fl{\alpha j} (d/\alpha q -d/q)} \sigma_{\fl{\alpha j}} /\gamma_{\fl{\alpha j}}^{(h)})^M \\
  &\le C''' \sum_{j\ge J} \fl{\alpha j}^{a_0 m_1} 2^{j(d- m_1\alpha (\zeta(h) - d/\alpha q -d/q -\epsilon_0))}.
 \end{align*}
 Letting $J$ going to $\infty$, (\ref{eq:def m0}) implies
 \[
  \leb (\{a \in \R^{m_1} : f+ a g \in F^{\alpha,C}(h)\})=0,
 \]
 hence the conclusion.
\end{proof}

\begin{Thm}
 Let $p,q,r,s\in [1,\infty]$ with $s\le q$, $\seq{\sigma}$ be an admissible sequence and $\seq{\gamma}^{(\cdot)}$ be a family of admissible sequences compatible with $\seq{\sigma}$. From the prevalence point of view, for almost every $f\in B^{\seq{\sigma}}_{r,s}$, $\ms_{p,q}$ is defined on $I=[\zeta^{-1}(-d/r), \zeta^{-1}(0)]$ and
 \[
  \ms_{p,q}(h) = d + r \zeta(h),
 \]
 for any $h\in I$.

 Moreover, for almost every $x_0\in \Rd$, we have $h_{p,q}(x_0)= \zeta^{-1}(0)$.
\end{Thm}
\begin{proof}
 We know that
 \[
  \{ f\in B^{\seq{\sigma}}_{r,s} : (\exists x_0\in E^\alpha: f\in \ep^{\seq{\sigma}^{(h)}}_{p,q}(x_0) \} \subset \bigcup_{l\in \N} F^{\alpha,l}(h), 
 \]
%
 so that, for any $\alpha\ge 1$ and any $h> h_*(\alpha)$, for almost every $f\in B^{\seq{\sigma}}_{r,s}$, we have, for every $x_0\in E^\alpha$, $h_{p,q}(x_0)\le h$.
 By countable intersection, we thus get that for almost every $f\in B^{\seq{\sigma}}_{r,s}$, we have, for every $x_0\in E^\alpha$, $h_{p,q}(x_0)\le h(\alpha)$. Let $f\in B^{\seq{\sigma}}_{r,s}$ be such that the preceding assertion holds.

 First, let us fix $\alpha\in (1,\infty)$; If $\seq{\alpha}$ is an increasing sequence of rational numbers converging to $\alpha$, the sequence $(E^{\alpha_n})_n$ is decreasing and $E^\alpha \subset \cup_n E^{\alpha_n}$. If $x_0$ belongs to $E^{\alpha_n}$, we have $h_{p,q}(x_0)\le h_*(\alpha_n)$ and thus $h_{p,q}(x_0)\le h_*(\alpha)$, for every $x_0\in E^\alpha$. Let $\mu_\alpha$ be a measure such that
 \begin{itemize}
  \item $\supp(\mu_\alpha) \subset E^\alpha$,
  \item $\mu_\alpha(E^\alpha) >0$,
  \item $\mu_\alpha(F)=0$ whenever $\ha(F)< d/\alpha$,
 \end{itemize}
 let us define
 \[
  F^\alpha := \{x_0\in [0,1]^d : h_{p,q}(x_0)< h_*(\alpha)\}
 \]
 and, for $n\in \N$,
 \[
  F_n^\alpha := \{x_0\in [0,1]^d : h_{p,q}(x_0)< h_*(\alpha)-1/n\}.
 \]
 For $n$ large enough, we have $h(\alpha)-1/n\ge -d/p$ and thus $\ha(F_n^\alpha) <d/\alpha$. Since $F^\alpha$ is included in a countable union of $\mu_\alpha$-measurable null sets, we have $\mu_\alpha(F^\alpha)=0$. As a consequence, we have
 \[
  \mu_\alpha(E^\alpha\setminus F^\alpha) \ge d + r \zeta (h_*(\alpha)).
 \]
 Since
 \[
  E^\alpha\setminus F^\alpha \subset \{ x_0\in [0,1]^d: h_{p,q}(x_0)=h_*(\alpha)\},
 \]
 we get
  \[
  \ms(h_*(\alpha)) = d + r \zeta (h_*(\alpha)).
 \]

 If $\alpha=\infty$, we know that $x_0\in E^\infty$ implies $h_{p,q}(x_0)\le h_*(\alpha_n)$ for any $n\in \N$ and thus $h_{p,q}(x_0)\le h_{\min}(r)$. As a consequence, the set
 \[
  \{ x_0\in [0,1]^d : h_{p,q}(x_0) = h_{\min}(r)\}
 \]
 is not empty.

 It remains to consider the case $\alpha=1$. In this case, $E^1=[0,1]^d$ and $\mu_1$ can be chosen to be the Lebesgue measure restricted on $[0,1]^d$. For $x_0\in E^1$, $h_{p,q}(x_0)\le h_*(1)$ and by the same argument as in the first case, we get
 \[
  \mu_1 (\{ x_0\in [0,1]^d : h_{p,q}(x_0) < h_*(1)\})=0,
 \]
 so that $E^1$ is equal to $E^1\setminus F^1$ almost everywhere.

 As the proof can be easily adapted to any translated of $[0,1]^d$, the conclusion follows by countable intersection.
\end{proof}

\begin{Thm}
 Let $p,q,r,s\in [1,\infty]$ with $s\le q$, $\seq{\sigma}$ be an admissible sequence and $\seq{\gamma}^{(\cdot)}$ be a family of admissible sequences compatible with $\seq{\sigma}$. Let $x_0$ be a point of $\Rd$; from the prevalence point of view, for almost every $f\in B^{\seq{\sigma}}_{r,s}$, we have $h_{p,q}(x_0)= \zeta^{-1}(-d/r)$.
\end{Thm}
\begin{proof}
 Given $n\in\N$, let us define the admissible sequence $\seq{\theta}^{(n)}$ by
 \[
  \theta^{(n)}_j
  := \frac{1}{\gamma^{(\zeta^{-1}(-d/r)+1/n)}_j} \frac{1}{(j+1)^{1+1/s}},
 \]
 $j\in \N$. We can now define the function $g^{(n)}$ which is a function whose wavelet coefficients are
 \[
  c^{(n)}_\lambda := \left\{\begin{tabular}{ll}
          $\theta^{(n)}_j$ & if $\lambda\in \Lambda_j$ and $\lambda = \lambda_j(x_0)$ \\
          $0$ & if $\lambda\in \Lambda_j$ and $\lambda \not= \lambda_j(x_0)$
         \end{tabular}\right..
 \]
 Since, for $n\in \N$, there exists $C_n>0$ such that
 \[
  (\sum_{\lambda \in \Lambda_j} (\sigma_j 2^{j d/r} |c^{(n)}_\lambda |)^r)^{1/r}
  \le C_n /(j+1)^{1+1/s},
 \]
 $g^{(n)}$ belongs to $B^{\seq{\sigma}}_{r,s}$.
 
 Let us fix $n_0\in \N$ and define
 \[
  F_{n_0} := \{ f\in B^{\seq{\sigma}}_{r,s} : \forall j\in \N \forall \lambda\in \Lambda_j\cap \K_{x_0}(2), |c_\lambda| \le n_0 \theta^{(n)}_j/j \}.
 \]
 As shown before, $F_{n_0}$ is a Borel set. For $f\in B^{\seq{\sigma}}_{r,s}$ and $a,a'\in \R$ satisfying $f+a g^{(n)}\in F_{n_0}$ and $f+a' g^{(n)}\in F_{n_0}$, we get
 \[
  |a-a'| \le 2n_0/j,
 \]
 so that the Lebesgue measure of $\{a\in \R : f + a g^{(n)}\in F_{n_0}\}$ vanishes, implying that $F_{n_0}$ is Haar-null. As we have
 \[
  \{ f\in B^{\seq{\sigma}}_{r,s} : f\in \ep^{\seq{\theta}^{(n)}}_{p,q}(x_0)\} \subset \bigcup_{l\in \N} F_{l},
 \]
 for almost every $f\in B^{\seq{\sigma}}_{r,s}$, we have $h_{p,q}(x_0)\le \zeta^{-1}(-d/r)+1/n$, which leads to the conclusion.
\end{proof}

\bibliography{pbes}{}
\bibliographystyle{plain}

\end{document}